\documentclass[a4paper,11pt]{article}

\usepackage{newcent}

\usepackage[utf8]{inputenc}
\usepackage[english]{babel}
\usepackage{graphicx}
\usepackage[T1]{fontenc}
\usepackage{amsmath}
\usepackage{amssymb} 
\usepackage{amsthm}

\usepackage{ifthen,amsfonts,graphicx,latexsym,algorithm,algorithmic,url,txfonts}

\usepackage{mathabx}
\usepackage{bbm} 
\usepackage{tikz}
\usepackage[numbers]{natbib}
\usepackage{titlesec}
\usepackage[textwidth=510pt,textheight=660pt]{geometry}

\usepackage{authblk} 

\usepackage{setspace}

\usepackage[pdfborder={0 0 0}]{hyperref}

\numberwithin{equation}{section} 

\newcommand{\R}{\mathbb{R}}
\newcommand{\N}{\mathbb{N}}
\newcommand{\Z}{\mathbb{Z}}

\newcommand{\Prob}{\mathcal{P}}
\newcommand{\Leb}{\mathcal{L}} 
\newcommand{\C}{\mathcal{C}} 
\newcommand{\A}{\mathcal{A}} 
\newcommand{\T}{\mathcal{T}} 

\DeclareMathOperator*{\esssup}{ess\,sup}

\newcommand{\ddr}{~\mathrm{d}} 

\newcommand{\1}{\mathbbm{1}} 

\newcommand{\dst}[1]{\displaystyle{#1}}

\newtheorem{theo}{Theorem}[section]
\newtheorem*{theo*}{Theorem}
\newtheorem{prop}[theo]{Proposition}
\newtheorem{crl}[theo]{Corollary}
\newtheorem{lm}[theo]{Lemma}
\newtheorem{defi}[theo]{Definition}
\newtheorem{remark}[theo]{Remark}
\newtheorem{asmp}{Assumption}

\newcommand{\intervalle}[4]{\mathopen{#1}#2
                            \mathclose{}\mathpunct{},#3
                            \mathclose{#4}}
\newcommand{\intervalleff}[2]{\intervalle{[}{#1}{#2}{]}}
\newcommand{\intervalleof}[2]{\intervalle{(}{#1}{#2}{]}}
\newcommand{\intervallefo}[2]{\intervalle{[}{#1}{#2}{)}}
\newcommand{\intervalleoo}[2]{\intervalle{(}{#1}{#2}{)}}

\begin{document}

\title{Optimal density evolution with congestion:\\ $L^\infty$ bounds via flow interchange techniques\\ and applications to variational Mean Field Games}
\author{Hugo Lavenant and Filippo Santambrogio}

\maketitle
\abstract{We consider minimization problems for curves of measure, with kinetic and potential energy and a congestion penalization, as in the functionals that appear in Mean Field Games with a variational structure. We prove $L^\infty$ regularity results for the optimal density, which can be applied to the rigorous derivations of equilibrium conditions at the level of each agent's trajectory, via time-discretization arguments, displacement convexity, and suitable Moser iterations. Similar $L^\infty$ results have already been found by P.-L. Lions in his course on Mean Field Games, using a proof based on the use of a (very degenerate) elliptic equation on the dual potential (the value function) $\varphi$, in the case where the initial and final density were prescribed (planning problem). Here the strategy is highly different, and allows for instance to prove local-in-time estimates without assumptions on the initial and final data, and to insert a potential in the dynamics.}

\tableofcontents

\section{Introduction}

The problem of optimal density evolution with congestion is a very natural question where an initial distribution $\rho_0$ of mass (particles, individuals\dots) is given, and it has to evolve from time $t=0$ to time $t=T$ by minimizing an overall energy. This typically involves its kinetic energy expenditure and a cost depending on congestion effects, i.e. on how much it is concentrated along its trajectory. Then, at time $t=T$, either the final configuration $\rho_T$ is prescribed, or a final cost depending on $\rho_T$ is also considered. 

When $\rho_T$ is fixed and no congestion effect is present, the only quantity to be minimized is the kinetic energy and this boils down to what is usually known as the dynamic formulation of the optimal transport problem, studied by Benamou and Brenier in \cite{Benamou2000}. From the fluid mechanics point of view, this model corresponds to that of particles of a pressureless gas moving without acceleration in straight lines, and without interaction with each other. From the geometric point of view, this variational problem consists in looking for a geodesic in the Wasserstein space $W_2$ (for references on optimal transport and Wasserstein spaces, see \cite{Villani2003,SantambrogioOTAM}). Inserting congestion effects corresponds to looking at deformed geodesics, i.e. curves which are optimal for other criteria which do not involve only their length (weighted lengths, length + penalizations\dots), and to add pressure terms in the corresponding gas equations. For instance, in \cite{Brancolini2006,Ambrosio2007} geodesics in the Wasserstein space for different weights, minimizing energies of the form $\int E(\rho(t))|\dot{\rho_t}|\ddr t$ were considered, including cases where $E$ penalized congestion. Yet, the case which is now the most studied is the one where a penalization on congestion is added to the kinetic energy, thus minimizing $\int (|\dot{\rho_t}|^2+E(\rho(t)))\ddr t$, as it was done in \cite{Buttazzo2009}. Since this is the kind of problems this paper will be devoted to, it is important to clarify precisely its form. We can either look for a curve $\rho:[0,T]\to \Prob(\Omega)$ which minimizes 
\begin{equation}\label{mini with rhorho}
\rho\mapsto \int_0^T \frac12 |\dot{\rho_t}|^2\ddr t+\int_0^T E(\rho_t)\ddr t+\Psi(\rho_T)
\end{equation}
with $\rho_0$ prescribed (here $ |\dot{\rho_t}|$ is the metric derivative, i.e. the speed of this curve for the distance $W_2$, see Section \ref{sec2.2}), or look for a pair $(\rho,\mathbf v)$ minimizing
\begin{equation}\label{mini with rho v}
\rho\mapsto \int_0^T\int_\Omega \frac12\rho_t|\mathbf{v}_t|^2\ddr t+\int_0^T E(\rho_t)\ddr t+\Psi(\rho_T)
\end{equation}
under the same constraint on $\rho_0$ and a differential constraint $\partial_t\rho_t+\nabla\cdot(\rho_t \mathbf{v}_t)=0$. Here $\Omega\subset\R^d$ is a bounded and connected domain and the continuity equation is satisfied in the weak sense on $[0,T]\times\R^d$ (which corresponds to imposing no-flux boundary conditions on $\partial\Omega$), The equivalence between the two formulations essentially comes from \cite{Benamou2000} and from the characterization of absolutely continuous curves in the Wasserstein space studied in \cite{Ambrosio2005}.The functional $E$ usually takes the form of an integral functional such as 
$$E(\rho):=\int_\Omega f(\rho(x))\ddr x+\int_\Omega V(x)\rho(x)\ddr x,$$
for a convex function $f$ and a given potential $V$ (where we identify measures with their densities; for the definition for measures which are not absolutely continuous, see Section \ref{sec2.3}). The final penalization $\Psi$ can be either a functional of the same form of $E$, or a constraint which prescribes $\rho_T$.

The interest for this minimization problem, which is already very natural in itself, has increased a lot after the introduction in 2006 of the theory of Mean Field Games (MFG) (introduced essentially at the same time by Lasry and Lions, \cite{Lasry2006-I,Lasry2006-II,Lasry2007}, and by Huang, Malham\'e and Caines \cite{Huang2006}). In the easiest version of these games, we consider a population of agents where everybody chooses its own trajectory, solving
\begin{equation}\label{cost}
\min\;\int_0^T \left(\frac{|x'(t)|^2}{2}+V(x(t))+g(\rho_t(x(t)))\right)\ddr t + \Psi(x(T)),
\end{equation}
with given initial point $x(0)$. Here $g$ is a given increasing function of the density $\rho_t$ at time $t$, i.e. an individual cost for each agent penalizing congested areas. The difficulty in the model is that every agent optimizes given the density of all agents $\rho_t$, but this density depends on turn on the choices of all the agents. An equilibrium problem arises, and we look for a Nash equilibrium in this continuum game (with infinitely many negligible players, who move continuously in time in a continuous space). This can be translated into a system of PDEs
\begin{equation*}\label{MFGgrho}
\begin{cases}-\partial_t\varphi+\frac{|\nabla\varphi|^2}{2}=V(x)+g(\rho),\\
		\partial_t \rho -\nabla\cdot (\rho \nabla \varphi)=0,\\
		\varphi(T,x)=\Psi(x),\quad\rho(0,x)=\rho_0(x).\end{cases}
\end{equation*}		
This forward-backward system is composed of a Hamilton-Jacobi equation for the value function $\varphi$ of the above optimization problem where the density $\rho$ appears at the right-hand side and of a continuity equation for $\rho$ which is advected by the vector field $\mathbf{v}=-\nabla\varphi$. Taking the gradient of the HJ equation gives a formula for the Lagrangian acceleration
$$\partial_t \mathbf{v}_t+(\mathbf{v}_t\cdot \nabla)\mathbf{v}_t=\nabla (V+g(\rho_t)),$$
where $g(\rho_t)$ plays the role of a pressure to be added to the potential $V$, as it is typical in compressible fluid mechanics.
Alternatively, the same equilibrium problem can be formulated in terms of a probability measure $Q$ on the set $H^1([0,T];\Omega)$ of paths valued in $\Omega$, defining $\rho_t=(e_t)_\# Q$ (where $e_t:H^1([0,T];\Omega)\to\Omega$ is the evaluation map at time $t$), and requiring $(e_0)_\#Q=\rho_0$ and that $Q$-a.e. curve is optimal for \eqref{cost} with this choice of $\rho_t$. For a general introduction to mean field games, other than the papers by Lasry-Lions and Huang-Malham\'e-Caines, the reader can consult the lecture notes by Cardaliaguet \cite{Cardaliaguet2010}, based on the lectures given by P.-L. Lions at Coll\`ege de France between 2006 and 2012 (\cite{LionsVIDEO}). In particular, for the simple model which is the object of this paper, and which is deterministic and first-order (no random effect in the motion of the agents, and no diffusion in the equations), we also refer to \cite{Cardaliaguet2015}.

The remarkable fact is that this class of equilibrium problem has a variational origin, and one can find an equilibrium by minimizing \eqref{mini with rho v} choosing $E(\rho):=\int f(\rho)+V(x)\rho$ with $f'=g$ (for a review on variational mean field games and on these questions, we refer to \cite{Benamou2016}). The optimality condition on the optimal $(\rho,\mathbf v)$ will indeed show that we have $\mathbf{v}=-\nabla\varphi$ where $\varphi$ solves the HJ part of \eqref{MFGgrho}, thus getting a solution of the system. The same can also be formally formulated in terms of probabilities $Q$ on the set of path. 

Yet, these considerations are essentially formal and not rigorous, so far. Indeed, the difficulty is the following: the function $h(t,x):=V(x)+g(\rho_t(x))$ is obtained from the density of a measure, and hence it is only defined a.e. Integrating it on a curve, as we do when we consider the action $\int_0^T h(t,x(t))\ddr t$ in \eqref{cost} has absolutely no meaning! Of course, it would be different if we could prove some regularity (for instance, continuity) on $\rho_t$. The question of the regularity in mean field games is a very challenging one and deserves high attention. In \cite{Cardaliaguet2016} a stategy to overcome this difficulty, taken from \cite{Ambrosio2009}, is used: indeed, it is sufficient to chose a suitable representative of $h$ to give a precise meaning to the integral of $h$ on a curve, and the correct choice is 
$$\hat h(t,x):=\limsup_{r\to 0} h_r(t,x):=\fint_{B(x,r)}h(t,y)\ddr y;$$
to prove that $Q$ is concentrated on optimal curves for $\hat h$ it is then enough to write estimates with $h_r$ and then pass to the limit as $r\to 0$. This requires an upper bound on $h_r$, and the natural assumption is to require that the maximal function $Mh:=\sup_r h_r$ is $L^1$ in space and time. Thanks to well-known results in harmonic analysis, $h\in L^1$ is not enough for this but $h\in L^m$ for $m>1$ is instead enough. Once integrability of $Mh$ is obtained, then one can say that the optimal measure $Q$ is concentrated on curves which minimize in \eqref{cost} in the class of curves $x(\cdot)$ such that $\int_0^T Mh(t,x(t))\ddr t<+\infty$. These curves are almost all curves in a suitable sense, thanks to the integrability properties of $Mh$ in space-time, but they are in general {\it not all} curves. 

It is interesting to observe that the strategy of \cite{Cardaliaguet2016} and \cite{Ambrosio2009} was first used in the framework of variational models for the incompressible Euler equation, in the sense of Brenier \cite{Brenier1989,Brenier1999}. Indeed, the problem of incompressible evolution has many similarities with the one of evolution with congestion effects, with the only difference that instead of penalizing high densities there is a constraint $\rho=1$. Also, the precise mean field game studied in \cite{Cardaliaguet2016} is of very similar nature, since it included the constraint $\rho\leq 1$. Moreover, the techniques used in \cite{Cardaliaguet2016} to prove this extra summability of $h$ come from the incompressible Euler framework: they are techniques based on convex duality taken from \cite{Brenier1999} and later improved in \cite{Ambrosio2008}, which allow, in this case, to prove $h\in L^2_{loc}((0,T);BV_{loc}(\Omega))$. In the framework of more standard mean field games (i.e. with density penalization instead of constraints), the same technique (presented in more generality on some simpler examples in \cite{San16a}) has been used in \cite{Prosinski2017} to prove $H^1$ regularity results on the density $\rho$. 

\bigskip

In the present paper, we present $L^\infty$ bounds on the optimal $\rho$. For applications to MFG, whenever $L^\infty$ results are available, it is possible to avoid all the assumptions on the maximal function $Mh$ and obtain optimality in the larger class of all competing curves. This explains the interest of these results for MFG, but of course the reader can easily guess that they are interesting in themselves for the variational problem. 

The question of the $L^\infty$ regularity of $\rho$ was already studied, in the MFG framework, by P.-L. Lions (see the second hour of the video of the lecture of November 27, 2011, in \cite{LionsVIDEO}), but the analysis was limited to global results when both $\rho_0$ and $\rho_T$ are fixed and $L^\infty$, and no potential $V$ is considered. The technique was essentially taken from degenerate elliptic PDEs (note that adapting from global to local results would be very difficult, without strong assumptions on the degeneracy and growth of the corresponding equation). Here what we do is different. The technique is based on the time-discretization of \eqref{mini with rhorho} in the form
\begin{equation}\label{timediscr}
\min\left\{ \sum_{k=1}^N \frac{W_2^2(\rho_{(k-1)\tau}, \rho_{k \tau})}{2 \tau} + \sum_{k=1}^{N-1} \tau E(\rho_{k\tau}) + \Psi(\rho_{N\tau})\right\},
\end{equation}
where $\tau = T/N$. The interesting fact is that, as a necessary optimality condition, each measure $\rho_{k\tau}$ with $0 < k<N$ minimizes
$$\rho\;\mapsto\;\frac{W_2^2(\rho_{(k-1)\tau}, \rho)}{2 \tau} + \frac{W_2^2(\rho_{(k+1)\tau}, \rho)}{2 \tau}+ \tau E(\rho),$$
which is very similar to what we see in the so-called Jordan-Kinderlehrer-Otto scheme for the gradient flow of the functional $E$ (see \cite{Jordan1998} and \cite{Ambrosio2005}). The main difference is that we have now two Wasserstein terms, one referring to the distance to the previous measure and one to the next one. Techniques from the JKO scheme can be used, and in particular the so-called {\it flow-interchange} technique (introduced in \cite{Matthes2009}).  Essentially, this technique consists in evaluating how much decreases another energy $U$ along the gradient flow of $E$. In the JKO framework, it is usually used to obtain estimates of the form
$$U(\rho_{k\tau})-U(\rho_{(k+1)\tau})\geq \tau \int \left(\mbox{something positive}\right),$$
which allows to say that $U$ is decreasing and to obtain integral estimates on the right hand side (r.h.s.) above. In this variational framework, which corresponds to a second-order-in-time equation, instead of monotonicity we obtain convexity:
$$\frac{U(\rho_{(k-1)\tau})+U(\rho_{(k+1)\tau})-2U(\rho_{k\tau})}{\tau^2}\geq \int \left(\mbox{something positive}\right)$$
(more precisely: the integral term in the r.h.s. is nonnegative if $V=0$ and extra lower-order terms appear in presence of a potential $V$).
This allows for instance to obtain convexity in time of all the quantities of the form $\int \rho_t^m(x)\ddr x$ when $V=0$ (a similar technique was used with similar results by the first author in \cite{Lavenant2017}). A global $L^\infty$ result if $\rho_0,\rho_T\in L^\infty$ are fixed is then easy to deduce in this case (actually, we will not even state it explicitly in this paper). Moreover, using the structure of the right-hand side and with tedious iterations inspired by Moser \cite{Moser1960}, we are also able to provide interior $L^\infty$ regularity independent of the boundary data, and regularity on intervals of the form $[t_1,T]$ under some assumptions on the penalization $\Psi$. This very result is, by the way, the natural one for MFG applications, and improves upon the results announced in  \cite{LionsVIDEO}. 

The paper is organized as follows: after this brief introduction Section 1 also contains a short summary of the main ideas of the proof, so that the reader does not get lost in the technical details. Then, in Section 2 we summarize the preliminaries about curves and functionals on the Wasserstein space, and give a precise statement for the variational problem we consider and the results we prove, distinguishing into two cases depending on the convexity of the congestion penalization $f$ (in terms of lower bounds on $f''$). In Section 3 we present and prove the estimates that are obtained in this framework via the flow interchange technique. These estimates allow to bound increasing $L^m$ norms of the solution, and in Section 4 we explain how to iterate in order to transform them into $L^\infty$ estimates on the limit of the discretized problems. This involves a technical difficulty, as one needs a reverse Jensen inequality in time (passing from $\int ||\rho_t||_{L^{\beta m}}^m \ddr t$ to $\left(\int ||\rho_t||_{L^{\beta m}}^{\beta m}\ddr t\right)^{1/\beta}$); this can be fixed because we already proved a convexity-like property for $t\mapsto  ||\rho_t||_{L^{\beta m}}^{\beta m}$ but is quite technical.  In Section 5 we detail how to pass to the limit from the time-discretization to the continuous problem, and in the Appendix we give a proof of the reverse Jensen inequality.

\subsection{Structure of the proof}

As the structure of the proof of $L^\infty$ bounds may be hidden behind the technical details, we sketch in this subsection the formal computations on which our main results rely. Let us consider the simplest case, the one where there is no interior potential, and let us not worry about the temporal boundary terms for the moment. The variational problem reads
\begin{equation*}
\min \left\{ \int_0^T \frac{1}{2} |\dot{\rho}_t|^2 \ddr t + \int_0^T \int_\Omega f(\rho_t(x)) \ddr x \ddr t + \Psi( \rho_T) \ : \ \rho : \intervalleff{0}{T} \to \Prob(\Omega),\;\rho_0=\overline{\rho_0} \right\}. 
\end{equation*} 
First, we consider the time-discretization in \eqref{timediscr} and, we apply the flow-interchange technique that we mentioned before (and that will be detailed later, see Section \ref{flow int sec 3}) to the functional $U=U_m$, where
\begin{equation*}
U_m(\rho) := \frac{1}{m(m-1)} \int_\Omega \rho(x)^m \ddr x,
\end{equation*}
when $m>1$ ($U_1(\rho)$ can be defined as the Boltzmann entropy of $\rho$, and the normalization constants are chosen for coherence with this case).
The flow interchange technique will allow to obtain an estimate of the form
\begin{equation}\label{discrete conv 1.1}
\frac{U_m(\rho_{(k-1)\tau})+U_m(\rho_{(k+1)\tau})-2U_m(\rho_{k\tau})}{\tau^2}\geq C(m) \int f''(\rho_{k\tau}) \rho_k^{m-1}|\nabla\rho_{k\tau}|^2\geq 0.
\end{equation}

This gives a discrete-in-time convex behavior for the quantity $U_m(\rho_t)$. In the case $f=0$, this is basically a restatement of McCann convexity principle \cite{Mccann1997}. It is easy to see that, if $\rho_0,\rho_T\in L^m$ (which is an assumption on $\rho_0$ and on $\Psi$), then automatically the same $L^m$ bound is satisfied by all measures $\rho_t$. If $\rho_0,\rho_T\in L^\infty$, the same easily passes to the limit as $m\to \infty$ thus providing $L^\infty$ bounds.

Moreover, the flow interchange applied to the last time step $k=N$ (with $N\tau=T$), gives a result of the form
\begin{equation} \label{equation_flow_interchange_Neumann}
\frac{U_m(\rho_{N\tau})-U_{m}(\rho_{(N-1)\tau})}{\tau}\leq b(m) U_m(\rho_{N\tau}),
\end{equation}
where the constant $b(m)$ depends on the penalization $\Psi$.
This acts as a sort of Neumann condition for the function $t\mapsto U_m(\rho_t)$ and allows to obtain the bound on $U_m(\rho_t)$ with the only assumption $\rho_0\in L^m$, with no need to assume the same for $\rho_T$. However, this can only be adapted to the limit $m\to\infty$ in the case where $\Psi$ has the form $\Psi(\rho):=\int g(\rho(x))\ddr x$ for a convex $g$, without potential terms, so that $b(m)=0$. Otherwise, the dependence of $b(m)$ upon $m$ prevents from letting $m\to\infty$.

Our paper includes $L^m$ and $L^\infty$ results which do not require assumptions on $\rho_0$, and which will be, of course, only of local nature on $(0,T]$. Our proof will look like Moser's proof of regularity for elliptic equations \cite{Moser1960}, as it will rely on a fine analysis of the growth (when $m \to + \infty$) of quantities of the form $\int_{T_1}^{T_2} \rho_t^m$. Indeed, one can guess from \eqref{discrete conv 1.1} that we may write (at the limit when $\tau\to 0$)
\begin{equation}
\label{equation_flow_interchange_heuristic}
\frac{d^2}{dt^2} U_m(\rho_t) \geqslant \int_\Omega |\nabla \rho_t|^2 \rho_t^{m-1} f''(\rho_t).
\end{equation}



To estimate more precisely the r.h.s. of \eqref{equation_flow_interchange_heuristic}, a natural assumption is $f''(s) \geqslant s^\alpha$ (with $\alpha$ which could be negative, of course): if this is the case, one can check that the integrand of the r.h.s. is larger than $|\nabla (\rho_t^{(m+1+\alpha)/2})|^2$ (up to a constant depending polynomially in $m$). Using the Sobolev injection $H^1 \hookrightarrow L^{2d/(d-2)}$, one can conclude (neglecting the $0$-order term of the $H^1$ norm of $\rho_t^{(m+1+\alpha)/2}$), with $1 < \beta < d/(d-2)$, that 
\begin{equation*}
C(m) \frac{d^2}{dt^2} U_m(\rho_t) \geqslant \left( \int_\Omega \rho_t^{\beta(m+1+\alpha)} \right)^{1/\beta}
\end{equation*} 
In the case $\alpha \geqslant -1$, we see that the r.h.s. is larger than $U_{\beta m}(\rho_t)^{1/ \beta}$. In other words, we have obtained a control of $U_{\beta m}(\rho)$ in terms of $U_m(\rho)$. Such a control can be iterated. If we take a positive cutoff function $\chi$ which is equal to $1$ on $\intervalleff{T_1 - \varepsilon}{T_2 + \varepsilon}$ and which is null outside $\intervalleff{T_1 - 2\varepsilon}{T_2 + 2\varepsilon}$, multiply \eqref{equation_flow_interchange_heuristic} by $\chi$ and integrate the left hand side (l.h.s.) by parts twice, we can say that 
\begin{equation*}
\int_{T_1 - \varepsilon}^{T_2 + \varepsilon} U_{\beta m}(\rho_t)^{1/ \beta} \ddr t \leqslant C(m, \varepsilon) \int_{T_1 - 2\varepsilon}^{T_2 + 2\varepsilon} U_m(\rho_t) \ddr t, 
\end{equation*}
where the constant $C(m, \varepsilon)$ grows not faster than a polynomial function of $m$ and $\varepsilon^{-1}$. We have to work a little bit more on the l.h.s. because we want to exchange the power $1/\beta$ and the integral sign, and unfortunately Jensen's inequality gives it the other way around. To this extent, we rely on the following observation: as the function $U_{\beta m}$ is convex (this can be seen in \eqref{equation_flow_interchange_heuristic}) and positive, it is bounded on $\intervalleff{T_1}{T_2}$ either by its values on $\intervalleff{T_1}{T_1 - \varepsilon}$ or on $\intervalleff{T_2}{T_2 + \varepsilon}$, thus we have a "reverse Jensen's inequality"
\begin{equation*}
\left( \int_{T_1}^{T_2} U_{\beta m}(\rho_t) \ddr t \right)^{1 / \beta} \leqslant \frac{(T_2 - T_1)^{1 / \beta}}{\varepsilon} \left( \int_{T_1 - \varepsilon}^{T_1} U_{\beta m}(\rho_t)^{1/ \beta} \ddr t + \int_{T_2}^{T_2 + \varepsilon} U_{\beta m}(\rho_t)^{1/ \beta}  \right).
\end{equation*}    
Combining this inequality with the estimation we have on the r.h.s., we deduce that 
\begin{equation*}
\left( \int_{T_1}^{T_2} U_{\beta m}(\rho_t) \ddr t \right)^{1 / \beta} \leqslant C(m, \varepsilon) \int_{T_1 - 2\varepsilon}^{T_2 + 2\varepsilon} U_m(\rho_t) \ddr t,
\end{equation*}
where the new constant $C(m, \varepsilon)$ has also a polynomial behavior in $m$ and $\varepsilon^{-1}$. This estimation is ready to be iterated. Indeed, setting $m_n := \beta^n m_0$ and $\varepsilon_n = 2^{-n} \varepsilon_0$, given the moderate growth of $C(m, \varepsilon)$, it is not difficult to conclude that 
\begin{equation*}
\limsup_{n \to + \infty} \left( \int_{T_1 - \varepsilon_n}^{T_2 + \varepsilon_n} U_{m_n}(\rho_t) \ddr t \right)^{1 / m_n} < + \infty.
\end{equation*} 
As the l.h.s. controls the $L^\infty$ norm of $\rho$ on $\intervalleff{T_1}{T_2} \times \Omega$, this is enough to conclude that $\rho$ is bounded locally in time and globally in space.

Let us comment the technical refinements and generalization of the above argument that are used in the present article: 
\begin{itemize}
\item[•] As we do not have enough time regularity to differentiate twice w.r.t. time, we decided to work with a time discretization of the problem. Hence, instead of\eqref{equation_flow_interchange_heuristic} we use \eqref{discrete conv 1.1}.
\item[•] If we add an interior potential, the r.h.s. of \eqref{equation_flow_interchange_heuristic} contains lower order terms that are controlled by the term involving $f''$. However, the sign of the l.h.s. is no longer known and the function $U_m$ is no longer convex but rather satisfies
\begin{equation*}
\frac{d^2}{dt^2} U_m(\rho_t) + \omega^2 U_m(\rho_t) \geqslant 0,
\end{equation*}
where $\omega$ grows linearly with $m$. In particular, the "reverse Jensen inequality" becomes more difficult to prove, but it is still doable. 
\item[•] With assumptions on the final penalization, the regularity can be extended to the final time. More precisely, if we assume that the final penalization is given by the sum of a potential term and a congestion term, then formally (and this can be proven by taking the limit $\tau \to 0$ of \eqref{equation_flow_interchange_Neumann}), 
\begin{equation}
\label{equation_flow_interchange_boudnary_heuristic}
\left. \frac{d}{dt} U_m(\rho_t) \right|_{t=T} \leqslant b(m) U_m(\rho_{T}),
\end{equation}   
where the constant $b(m)$ depends on the potential and can be taken equal to $0$ if there is no potential. This inequality enables to control the value of $U_m$ at the boundary $t = T$ by its values in the interior. Thus the same kind of iterations can be performed and gives $L^\infty$ regularity up to the boundary.  
\item[•] If $\alpha < -1$, we only have a control of $U_m$ by $U_{\beta(m+1+\alpha)}$. Thus we must start the iterative procedure with a value $m$ such that $m < \beta (m+1+\alpha)$, i.e. we must impose \emph{a priori} some $L^m$ regularity on $\rho$ (with a $m$ which depends on $\alpha$ and $\beta$, the latter depending itself only on the dimension of the ambient space). Such a regularity is imposed by assuming that $\rho_0$ (which is fixed) is in $L^m(\Omega)$ and that the boundary penalization in $t=T$ is the sum of a potential and a congestion term. Indeed, if this is the case, the boundary condition \eqref{equation_flow_interchange_boudnary_heuristic} combined with the interior estimate \eqref{equation_flow_interchange_heuristic} shows that if $T$ is small enough (given the potentials and the congestion function $f$), the $L^m$ norm of $\rho$ on $\intervalleff{0}{T} \times \Omega$ must be bounded.  
\end{itemize} 

\section{Notation and presentation of the optimal density evolution problem}

In all the sequel, $\Omega$ will denote the closure of an open bounded convex domain of $\R^d$ with smooth boundary. To avoid normalization constants, we will assume that its Lebesgue measure is $1$. The generalization to the case where $\Omega$ is the $d$-dimensional torus is straightforward and we do not address it explicitly. The space of probability measures on $\Omega$ will be denoted by $\Prob(\Omega)$. The Lebesgue measure restricted to $\Omega$, which is therefore a probability measure, will be denoted by $\Leb$. The space $\Prob(\Omega)$ is endowed with the weak-* topology, i.e. the topology coming from the duality with $C(\Omega)$ (the continuous functions from $\Omega$ valued in $\R$).  

\subsection{The Wasserstein space}

The space $\Prob(\Omega)$ of probability measures on $\Omega$ is endowed with the Wasserstein distance: if $\mu$ and $\nu$ are two elements of $\Prob(\Omega)$, the $2$-Wasserstein distance $W_2(\mu, \nu)$ between $\mu$ and $\nu$ is defined by 
\begin{equation}
\label{equation_definition_Wasserstein_distance}
W_2(\mu, \nu) := \sqrt{ \min \left\{ \int_{\Omega \times \Omega} |x-y|^2 \ddr \gamma(x,y) \ :  \ \gamma \in \Prob(\Omega \times \Omega) \text{ and } \pi_0 \# \gamma = \mu, \ \pi_1 \# \gamma = \nu \right\}  }.
\end{equation} 
In the formula above, $\pi_0$ and $\pi_1 : \Omega \times \Omega \to \Omega$ stand for the projections on respectively the first and second component of $\Omega \times \Omega$. If $T : X \to Y$ is a measurable application and $\mu$ is a measure on $X$, then the image measure of $\mu$ by $T$, denoted by $T \# \mu$, is the measure defined on $Y$ by $(T \# \mu)(B) = \mu(T^{-1}(B))$ for any measurable set $B \subset Y$. It can also be defined by 
\begin{equation*}
\int_Y a(y) \ddr (T \# \mu)(y) := \int_X a(T(x)) \ddr \mu(x), 
\end{equation*} 
this identity being valid as soon as $a : Y \to \R$ is any integrable function. For general results about optimal transport, the reader might refer to \cite{Villani2003} or \cite{SantambrogioOTAM}. We recall that $W_2$ admits a dual formulation: for any $\mu, \nu \in \Prob(\Omega)$, 
\begin{equation}
\label{equation_defintion_Wasserstein_dual}
W_2(\mu, \nu) = \sqrt{ \max \left\{ \int \varphi \ddr \mu + \int \varphi^{c} \ddr \nu \ : \ \varphi \in C(\Omega) \right\} },
\end{equation} 
where $\varphi^{c}(y) := \inf_{x \in \Omega} ( |x-y|^2 - \varphi(x))$ for any $y \in \Omega$. A function $\varphi\in \C(\Omega)$ which is optimal in \eqref{equation_defintion_Wasserstein_dual} is called a Kantorovitch potential for the transport from $\mu$ to $\nu$. The following result, giving the derivative of the Wasserstein distance, can be found in \cite[Propositions 7.18 and 7.19]{SantambrogioOTAM}. 

\begin{prop}
\label{proposition_first_variation_W2}
Let $\mu, \nu \in \Prob(\Omega)$ and assume that $\mu$ is absolutely continuous w.r.t. $\Leb$ and that its density is strictly positive a.e. Then there exists a unique Kantorovitch potential $\varphi$ for the transport from $\mu$ to $\nu$. Moreover, $\varphi$ is Lipschitz and if $\tilde{\mu} \in \Prob(\Omega) \cap L^\infty(\Omega)$, then 
\begin{equation*}
\lim_{\varepsilon \to 0} \frac{W_2^2( (1-\varepsilon) \mu + \varepsilon \tilde{\mu} , \nu) - W_2^2(\mu, \nu)}{\varepsilon} = \int_\Omega \varphi \ddr (\tilde{\mu} - \mu).
\end{equation*}
\end{prop}

We recall that $W_2$ defines a metric on $\Prob(\Omega)$ that metrizes the weak-* topology. Therefore, thanks to Prokhorov Theorem, the space $(\Prob(\Omega), W_2)$ is a compact metric space. We also recall that $(\Prob(\Omega), W_2)$ is a geodesic space. If $\mu$ and $\nu$ are probability measures such that $\mu$ admits a strictly positive density w.r.t. $\Leb$, then there exists a unique constant-speed geodesic $\rho : \intervalleff{0}{1} \to \Prob(\Omega)$ joining $\mu$ to $\nu$ and it is given by 
\begin{equation*}
\rho(t) = (\mathrm{Id} - t \nabla \varphi) \# \mu,
\end{equation*} 
where $\varphi$ is the unique Kantorovitch potential for the transport from $\mu$ to $\nu$. 

We will need to define functionals of the form $\mu \in \Prob(\Omega) \mapsto \int_\Omega h(\mu) \ddr \Leb$. To this extent, we rely on the following proposition (see \cite[Chapter 7]{SantambrogioOTAM}; see also \cite{Bouchitte1990} for the most advanced results on the semicontinuity of this kind of functionals on measures)

\begin{prop}
\label{proposition_continuity_functionals}
Let $h : \intervallefo{0}{+ \infty} \to \R$ be a convex function bounded from below. Let $h'(+ \infty) \in \intervalleof{- \infty}{+ \infty}$ be the limit of $h'(t)$ as $t \to + \infty$. Then, the functional
\begin{equation}\label{defi functional h}
\rho \in \Prob(\Omega) \mapsto \int_\Omega h(\rho^{ac}) + h'(+ \infty) \rho^{sing}(\Omega),   
\end{equation} 
(where $\rho =: \rho^{ac} \Leb + \rho^{sing} $ is the decomposition of $\rho$ as an absolutely continuous part $\rho^{ac} \Leb$ and a singular part $\rho^{sing}$ w.r.t. $\Leb$) is convex and l.s.c.
\end{prop} 

\noindent In particular, we will make a strong use of the following functionals. 

\begin{defi}
For any $m \geqslant 1$, we define $u_m : \intervallefo{0}{+ \infty} \to \R$ for any $t \geqslant 0$ through
\begin{equation*}
u_m(t) := \begin{cases}
t \ln t + 1 & \text{if } m =1 \\
\dst{\frac{t^{m}}{m(m-1)}} & \text{if } m > 1
\end{cases}.
\end{equation*} 
For any $m \geqslant 1$, the functional $U_m : \Prob(\Omega) \to \R$ is defined, for $\rho \in \Prob(\Omega)$, via
\begin{equation*}
U_m(\rho) := \begin{cases}
\dst{\int_\Omega u_m(\rho) } & \text{if } \rho \text{ is absolutely continuous w.r.t. } \Leb \\
+ \infty & \text{else} 
\end{cases}.
\end{equation*}
\end{defi}

\noindent One can notice that $u_m''(t) = t^{m-2}$ for any $m \geqslant 1$ and any $t > 0$, hence the functions $u_m$ are convex for all $m$. One can also notice that $U_1$ is (up to an additive constant) the entropy w.r.t. $\Leb$.  Moreover, some useful properties of $U_m$ are summarized below. 

\begin{prop}
For any $m \geqslant 1$, 
\begin{enumerate}
\item One has $m^2 U_m \geqslant 1$. 
\item The functional $U_m$ is convex and l.s.c. 
\item The functional $U_m$ is geodesically convex: it is convex along every constant-speed geodesic of $(\Prob(\Omega), W_2)$.
\end{enumerate}
\end{prop}

\begin{proof}
The first point derives from Jensen's inequality. The second point is an application of Proposition \ref{proposition_continuity_functionals}. To prove the third point, recall that $\Omega$ is convex: thus it is enough to check that the functions $u_m$ satisfy McCann's conditions (see \cite{Mccann1997} or \cite[Theorem 5.15]{Villani2003}), which is the case. 
\end{proof}

\subsection{Absolutely continuous curves in the Wasserstein space}\label{sec2.2}

We will denote by $\Gamma$ the space of continuous curves from $\intervalleff{0}{T}$ to $\Prob(\Omega)$.  This space will be equipped with the distance $d$ of the uniform convergence, i.e.
\begin{equation*}
d(\rho^1, \rho^2) := \max_{t \in \intervalleff{0}{T}} W_2( \rho^1(t), \rho^2(t) ).
\end{equation*} 

Following \cite[Definition 1.1.1]{Ambrosio2005}, we will use the following definition. 

\begin{defi}
We say that a curve $\rho \in \Gamma$ is $2$-absolutely continuous if there exists a function $\lambda \in L^2(\intervalleff{0}{T})$ such that, for every $0 \leqslant t \leqslant s \leqslant T $,  
\begin{equation*}
W_2(\rho_t, \rho_s) \leqslant \int_t^s \lambda(r) \ddr r.
\end{equation*}
\end{defi}

\noindent The main interest of this notion lies in the following theorem that we recall. 

\begin{theo}
If $\rho \in \Gamma$ is a $2$-absolutely continuous curve, then the quantity 
\begin{equation*}
|\dot{\rho}_t| := \lim_{h \to 0} \frac{W_2(\rho_{t+h}, \rho_t)}{h}
\end{equation*}
exists and is finite for a.e. $t$. Moreover, 
\begin{equation}
\label{equation_representation_A_sup}
\int_0^T |\dot{\rho}_t|^2 \ddr t = \sup_{N \geqslant 2} \ \ \sup_{0 \leqslant t_1 < t_2 < \ldots < t_N \leqslant T} \ \ \sum_{k=2}^{N} \frac{W_2^2(\rho_{t_{k-1}}, \rho_{t_k})}{t_k - t_{k-1}}.
\end{equation}
\end{theo}

\begin{proof}
The first part is just \cite[Theorem 1.1.2]{Ambrosio2005}. The proof of the representation formula \eqref{equation_representation_A_sup} can easily be obtained by adapting the proof of \cite[Theorem 4.1.6]{Ambrosio2003}.  
\end{proof}

The quantity $|\dot{\rho}_t|$ is called the metric derivative of the curve $\rho$ and heuristically corresponds to the norm of the derivative of $\rho$ at time $t$ in the metric space $(\Prob(\Omega), W_2)$. Thus, the quantity $\int_0^T |\dot{\rho}_t|^2 \ddr t$ behaves like a $H^1$ norm. In particular, we have the following.

\begin{prop}
\label{proposition_properties_action}
The function $\rho \in \Gamma \mapsto \int_0^T |\dot{\rho}_t|^2 \ddr t$ is l.s.c., convex, and its sublevel sets are compact. 
\end{prop}

\begin{proof}
The lower semi-continuity and convexity are a consequence of the representation formula \eqref{equation_representation_A_sup} (because the square of the Wasserstein distance is a continuous convex function of its two arguments, see \cite[Chapter 7]{SantambrogioOTAM}).

Moreover if $\rho \in \Gamma$ is a curve with finite action and $s < t$, then, again with \eqref{equation_representation_A_sup}, one can see that $W_2(\rho_s, \rho_t) \leqslant \sqrt{\int_0^T |\dot{\rho}_t|^2 \ddr t} \sqrt{t-s}$. This shows that the sublevel sets of $\int_0^T |\dot{\rho}_t|^2 \ddr t$ are uniformly equicontinuous, therefore they are relatively compact thanks to Ascoli-Arzela's theorem. As we know moreover that the sublevel sets are closed (by the lower semi-continuity we just proved), we can conclude that they are compact.
\end{proof}

\subsection{Continuous and discrete problems}\label{sec2.3}

In all the sequel, we will make the following assumptions: 

\begin{enumerate}
\item Recall that $\Omega$ is the closure of an open convex bounded domain with smooth boundary. 
\item We assume that $f : \intervallefo{0}{+ \infty} \to \R$ is a strictly convex function, bounded from below and $C^2$ on $\intervalleoo{0}{+ \infty}$. We define the congestion penalization $F$ by, for any $\rho \in \Prob(\Omega)$, 
\begin{equation*}
F(\rho) := \int_\Omega f(\rho^{ac}) + f'(+ \infty) \rho^{sing}(\Omega),
\end{equation*} 
where $\rho =: \rho^{ac} \Leb + \rho^{sing} $ is the decomposition of $\rho$ as an absolutely continuous part $\rho^{ac}$ (identified with its density) and a singular part $\rho^{sing}$ w.r.t. $\Leb$. Thanks to Proposition \ref{proposition_continuity_functionals}, we know that $F$ is a convex l.s.c. functional on $\Prob(\Omega)$. 
\item We assume that $V : \Omega \to \R$ is a Lipschitz function. 
\item We assume that $\Psi : \Prob(\Omega) \to \R$ is a l.s.c. and convex functional, bounded from below.
\end{enumerate}

\noindent We will consider variational problems with a running cost of the form $\rho\mapsto E(\rho):=F(\rho) + \int_\Omega V \ddr \rho$, while $\Psi$ will penalize the final density, and the initial one will be prescribed. 

\begin{defi}
We define the the functional $\A : \Gamma \to \R$ by 
\begin{equation*}
\A(\rho) := \int_0^T \frac{1}{2} |\dot{\rho}_t|^2 \ddr t + \int_0^T E(\rho_t) \ddr t + \Psi(\rho_T).
\end{equation*}
We state the continuous problem as 
\begin{equation}
\label{equation_continuous problem}
\tag{ContPb}
\min \{ \A(\rho) \ : \ \rho \in \Gamma,\;\rho_0=\overline{\rho_0} \}.
\end{equation}
A curve $\rho$ that minimizes $\A$ will be called a solution of the continuous problem. 
\end{defi}

\begin{prop}
Let us assume that there exists $\rho \in \Gamma$ with $\rho_0 = \overline{\rho_0}$ such that $\A(\rho) < + \infty$. Then the problem \eqref{equation_continuous problem} admits a unique solution.
\end{prop}
    
\begin{proof}
The functional $\A$ is the sum of l.s.c., convex and bounded functionals. Moreover, as $\A(\rho) \geqslant \int_0^T \frac{1}{2} |\dot{\rho}_t|^2 \ddr t -C$ (where $C$ depends on the lower bounds of $f, V$ and $\Psi$), we see (thanks to Proposition \ref{proposition_properties_action}) that the sublevel sets of $\A$ are compact. The existence of a solution to \eqref{equation_continuous problem} follows from the direct method of calculus of variations. 

To prove uniqueness, we need to prove that $\A$ is strictly convex. If $\rho^1$ and $\rho^2$ are two distinct minimizers of $\A$, we define $\rho : = (\rho^1 + \rho^2)/2$. As $\rho^1$ and $\rho^2$ are distinct, by continuity there exists $T_1 < T_2$ such that $\rho^1_t$ and $\rho^2_t$ differ for every $t \in \intervalleff{T_1}{T_2}$. In particular, for any $t \in \intervalleff{T_1}{T_2}$, by strict convexity of $F$, $F(\rho) < (F(\rho^1) + F(\rho^2))/2$. Thus, 
\begin{equation*}
\int_0^T F(\rho_t) \ddr t < \frac{1}{2} \int_0^T F(\rho^1_t) \ddr t + \frac{1}{2} \int_0^T F(\rho^2_t) \ddr t.
\end{equation*} 
As all the other terms appearing in $\A$ are convex, one concludes that $\A(\rho) <  ( \A(\rho^1) + \A(\rho^2))/2$, which contradicts the optimality of $\rho^1$ and $\rho^2$. 
\end{proof}     

\noindent In order to get the $L^\infty$ bounds, we will consider two different cases (strong and weak congestion), depending on the second derivative of $f$. This allows to quantify how much $F$ penalizes concentrated measures.

\begin{asmp}[strong congestion]
\label{assumption_strong_congestion}
There exists $\alpha \geqslant -1$ and $C_f > 0$ such that $f''(t) \geqslant C_f t^{\alpha}$ for any $t > 0$.
\end{asmp} 
\begin{asmp}[strong congestion-variant]
\label{assumption_strong_congestion-variant}
There exist $\alpha \geqslant -1$, $t_0 > 0$ and $C_f > 0$ such that $f''(t) \geqslant C_f t^{\alpha}$ for any $t \geq t_0$.
\end{asmp} 

\noindent In particular, integrating twice, we see that under either of the above assumptions, for $\rho \in \Prob(\Omega)$ we have $U_{\alpha + 2}(\rho) \leqslant C_f F(\rho) + C$, where $C$ is a constant that depends on $f$ (but not on $\rho$). One can also see that $f'(+ \infty) = + \infty$. The function $u_m$ is the typical example of a function satisfying Assumption \ref{assumption_strong_congestion} with $\alpha = m-2$.  To produce functions satisfying Assumption \ref{assumption_strong_congestion-variant} but not Assumption \ref{assumption_strong_congestion}, think at $f(t)=\sqrt{1 + t^4}$ (if we try to satisfy Assumption \ref{assumption_strong_congestion} we need $\alpha\leq 0$ for large $t$, and $\alpha\geq 2$ for small $t$) or at $f(t)=(t-1)_+^2$ (the difference between these two examples is that in the first case on could choose an aribtrary $t_0>0$, while in the second it is necessary to use $t_0\geq 1$).

\begin{asmp}[weak congestion]
\label{assumption_weak_congestion}
There exist $\alpha < -1$, $t_0 > 0$ and $C_f > 0$ such that $f''(t) \geqslant C_f t^{\alpha}$ for any $t \geqslant t_0$. 
\end{asmp}

\noindent For example, $f(t) := \sqrt{1 + t^2}$ satisfies $f''(t) \geqslant C_f t^{\alpha}$ for $t \geqslant 1$ with $\alpha = - 3$. 

\begin{asmp}[higher regularity of the potential]
\label{assumption_potential}
The potential $V$ is of class $C^{1,1}$ (it is $C^1$ and its gradient is Lipschitz) and $\nabla V \cdot \mathbf{n} \geqslant 0$ on $\partial \Omega$, where $\mathbf n$ is the outward normal to $\Omega$. 
\end{asmp}

\noindent We will see that only Assumption \ref{assumption_strong_congestion}, where we require a control of $f''$ everywhere, allows to deal with Lipschitz potentials, while in general we will need the use of Assumption \ref{assumption_potential}. The condition $\nabla V \cdot \mathbf{n} \geqslant 0$ on $\partial \Omega$ can be interpreted by the fact that the minimum of $V$ is reached in the interior of $\Omega$: it prevents the mass of $\rho$ to concentrate on the boundaries.

\begin{asmp}[final penalization]
\label{assumption_temporal_bc}
The penalization $\Psi$ is of the following form
\begin{equation*}
\Psi( \rho_T) = \begin{cases} 
\dst{\int_\Omega g(\rho_T) + \int_\Omega W \ddr \rho_T} & \text{if } \rho_T \text{ is absolutely continuous w.r.t. } \Leb  \\
+ \infty & \text{if } \rho_T \text{ is singular w.r.t. } \Leb,
\end{cases} 
\end{equation*}   
where $g : \intervallefo{0}{+ \infty} \to \R$ is a convex and superlinear (i.e. $g'(+ \infty) = + \infty$) function, bounded from below, and $W : \Omega \to \R$ is a potential of class $C^{1,1}$ satisfying $\nabla W \cdot \mathbf{n} \geqslant 0$ on $\partial \Omega$.
\end{asmp}


The mains results of this paper can be stated as follows. 

\begin{theo}[strong congestion, interior regularity]
\label{theorem_L_infty_bounds_strong}
Suppose that either Assumption \ref{assumption_strong_congestion} holds or Assumption \ref{assumption_strong_congestion-variant} and \ref{assumption_potential} hold, and that $\A(\rho) < + \infty$ for some $\rho \in \Gamma$ with $\rho_0=\overline{\rho_0}$. Let $\rho$ be the unique solution to \eqref{equation_continuous problem}. Then for any $0 < T_1 < T_2 < T$, the restriction of $\rho$ to $\intervalleff{T_1}{T_2}$ belongs to $L^\infty(\intervalleff{T_1}{T_2} \times \Omega)$.  
\end{theo}

\begin{theo}[strong congestion, boundary regularity]
\label{theorem_L_infty_bounds_boundary}
Suppose that either Assumption \ref{assumption_strong_congestion} holds or Assumption \ref{assumption_strong_congestion-variant} and \ref{assumption_potential} hold, and that Assumption  \ref{assumption_temporal_bc} holds as well, and that $\A(\rho) < + \infty$ for some $\rho \in \Gamma$ with $\rho_0=\overline{\rho_0}$. Let $\rho$ be the unique solution to \eqref{equation_continuous problem}. Then, for any $0 < T_1 < T$, the restriction of $\rho$ to $\intervalleff{T_1}{T}$ belongs to $L^\infty(\intervalleff{T_1}{T} \times \Omega)$.    
\end{theo}

\begin{theo}[weak congestion case]
\label{theorem_L_infty_bounds_weak}
Suppose Assumptions \ref{assumption_temporal_bc}, \ref{assumption_weak_congestion} and \ref{assumption_potential} hold and that $\A(\rho) < + \infty$ for some $\rho \in \Gamma$ with $\rho_0=\overline{\rho_0}$. We assume that the prescribed initial measure $\overline{\rho_0}$ satisfies $\overline{\rho_0}\in L^{m_0}$ with $m_0 > d |\alpha + 1|/2$ and $F(\overline{\rho_0})<+\infty$, and that $T$ is small enough (smaller than a constant that depends on $f, g, V, W$ and $\overline{\rho_0}$). Let $\rho$ be the unique solution to \eqref{equation_continuous problem}. Then $\rho \in L^{m_0}(\intervalleff{0}{T} \times \Omega)$ and for any $0 < T_1 < T$, the restriction of $\rho$ to $\intervalleff{T_1}{T}$ belongs to $L^\infty(\intervalleff{T_1}{T} \times \Omega)$.   
\end{theo}

The rest of the paper is devoted to the proof of these theorems. In particular, we will always assume in the sequel that there exists $\rho \in \Gamma$ with $\rho_0 = \overline{\rho_0}$ such that $\A(\rho) < + \infty$. In order to prove these theorems, we will introduce a discrete (in time) variational problem that will approximate the continuous one. For this problem, we will be able to show the existence of a unique smooth (in space) solution and write down the optimality conditions. From these optimality conditions, we will be able to derive a \emph{flow interchange} estimate whose iteration will give uniform (in the approximation parameters, and in $p$) $L^p$ estimates. 

\bigskip

Let us introduce the discrete problem here. We will use two approximations parameters: 
\begin{itemize}
\item[•] $N +1 \geqslant 2$ will denote the number of time steps. We will write $\tau := T/N$ for the distance between two time steps. The set $T^N$ will stand for the set of all time steps, namely 
\begin{equation*}
T^N := \left\{ k \tau ; \ k = 0, 1, \ldots, N \right\}. 
\end{equation*}
We set $\Gamma_N := \Prob(\Omega)^{T^N} \simeq \Prob(\Omega)^{N+1}$ (i.e. an element $\rho \in \Gamma_N$ is a $N+1$-uplet $(\rho_0, \rho_\tau, \ldots, \rho_{k \tau}, \ldots, \rho_T)$ of probability measures indexed by $T^N$). A natural discretization of the action of a curve is 
\begin{equation*}
\int_0^T \frac{1}{2} |\dot{\rho}_t|^2 \ddr t \simeq  \sum_{k=1}^N \frac{W_2^2(\rho_{(k-1) \tau}, \rho_{k \tau})}{2 \tau}.
\end{equation*} 
\item[•] We will also add a (vanishing) entropic penalization (recall that $U_1$ denotes the entropy w.r.t. $\Leb$). It will ensure that the solution of the discrete problem is smooth. The penalization will be a discretized version of 
\begin{equation*}
\lambda \int_0^T U_1(\rho_t) \ddr t,
\end{equation*} 
where $\lambda$ is a parameter that will be sent $0$. 
\end{itemize} 

\noindent Let us state formally our problem. We fix $N \geqslant 1$ ($\tau := T/N$) and $\lambda > 0$, and we set $\lambda_N = \lambda$ if Assumption \ref{assumption_temporal_bc} is satisfied, $\lambda_N = 0$ otherwise. We define $\A^{N, \lambda} : \Gamma_N \to \R$ by
\begin{equation*}
\A^{N, \lambda}(\rho) := \sum_{k=1}^N \frac{W_2^2(\rho_{(k-1) \tau}, \rho_{k \tau})}{2 \tau} + \sum_{k=1}^{N-1} \tau \left( E(\rho_{k\tau}) + \lambda U_1(\rho_{k \tau})  \right) + \Psi(\rho_T) + \lambda_N U_1(\rho_T).  
\end{equation*}
This means that in the case of Assumption \ref{assumption_temporal_bc} we penalize $\rho_T$ by $\int_\Omega g(\rho_T) + \lambda U_1(\rho_T) + \int_\Omega W \ddr \rho_T$, while we do not modify the boundary condition otherwise (the reason for not always adding $\lambda U_1(\rho_T)$ lies in the possibility of having a prescribed value for $\rho_T$ with infinite entropy). In all the cases, we enforce strictly $\rho_0 = \overline{\rho_0}$. The discrete minimization problem reads 
\begin{equation}
\label{equation_discrete_problem}
\tag{DiscrPb}
\min \{ \A^{N, \lambda}(\rho) \ : \ \rho \in \Gamma_N, \;\rho_0=\overline{\rho_0}\},
\end{equation}
and a $\rho \in \Gamma_N$ which minimizes $\A^{N, \lambda}$ will be called a solution of \eqref{equation_discrete_problem}. 

\begin{theo}
For any $N \geqslant 1$ and any $\lambda > 0$, the discrete problem \eqref{equation_discrete_problem} admits a solution.
\end{theo}

\begin{proof}
The functional $\A^{N, \lambda}$ is a sum of convex and l.s.c. functionals, bounded from below, hence it is itself convex, l.s.c. and bounded from below. Moreover, the space $\Gamma_N = \Prob(\Omega)^{N+1}$ is compact (for the weak-* convergence). Thus, to use the direct method of calculus of variations, it is enough to show that $A^{N, \lambda}(\rho) < + \infty$ for some $\rho \in \Gamma_N$. 

This is easy in this discrete framework: just take $\rho_{k \tau} = \Leb$ if $k \in \{ 1,2,\ldots, N-1 \}$, $\rho_0=\overline{\rho_0}$ and $\rho_{N\tau}$ equal to an arbitrary measure $\rho$ such that $\Psi(\rho)+\lambda_N U_1(\rho) < + \infty$. \end{proof}

\begin{remark}
We did not adress the uniqueness of the minimizer in the above problem since we do not really care about it, but indeed it also holds. Indeed, the strict convexity of $F$ (or the term $\lambda U_1$ that we added) guarantees uniqueness of $\rho_{k\tau}$ for all $k\leqslant N-1$. The uniqueness of the last measure (which cannot be deducted from strict convexity for an arbitrary functional $\Psi$, as we do not always add a term of the form $\lambda U_1(\rho_T)$) can be obtained from the strict convexity of the last Wasserstein distance term $\rho\mapsto W_2^2(\rho,\rho_{(N-1)\tau})$, as $\rho_{(N-1)\tau}$ is absolutely continuous (see \cite[Proposition 7.19]{SantambrogioOTAM}). 
\end{remark}

In all the following, for any $N \geqslant 1$ and $\lambda > 0$, we denote by $\bar{\rho}^{N, \lambda} \in \Gamma_N$ the unique solution of \eqref{equation_discrete_problem} with parameters $N$ and $\lambda$. Moreover, In all the sequel, we fix $1 < \beta < d/(d-2)$. It is well known that the space $H^1(\Omega)$ is continuously embedded into $L^{2 \beta}(\Omega)$. Moreover, in the case where the assumptions of Theorem \ref{theorem_L_infty_bounds_weak} are satisfied, we choose $\beta$ in such a way that 
\begin{equation}
\label{equation_relation_beta_m0}
\frac{\beta}{\beta - 1} m_0 >  |\alpha + 1|.
\end{equation}

\section{Flow interchange estimate}\label{flow int sec 3}

\subsection{Interior flow interchange}

In this subsection, we study the optimality conditions of \eqref{equation_discrete_problem} away from  the temporal boundaries. We fix for the rest of the subsection $N \geqslant 1$, $0 < \lambda \leqslant 1$ and $0 < k < N$, and we use the shortcut $\bar{\rho} := \bar{\rho}^{N, \lambda}_{k \tau}$. Let us also denote $\mu := \bar{\rho}^{N, \lambda}_{(k-1) \tau}$ and $\nu := \bar{\rho}^{N, \lambda}_{(k+1) \tau}$. As $\bar{\rho}^{N, \lambda}$ is a solution of the discrete problem, we know that $\bar{\rho}$ is a minimizer (among all probability measures) of 
\begin{equation*}
\rho \mapsto \frac{W_2^2(\mu, \rho) + W_2^2(\rho, \nu)}{2 \tau} + \tau \left( F(\rho) + \lambda U_1(\rho) + \int_\Omega V \ddr \rho \right).
\end{equation*} 
In particular, we know that $U_1(\bar{\rho}) < + \infty$, thus $\bar{\rho}$ is absolutely continuous w.r.t. $\Leb$.

\begin{lm}
\label{lemma_positivity_rho}
The density $\bar{\rho}$ is strictly positive a.e.
\end{lm} 

\begin{proof}
For $0 < \varepsilon < 1$, we define $\rho_\varepsilon := (1 - \varepsilon) \bar{\rho} + \varepsilon \Leb$. As $\Leb$ is a probability measure, we know that $\rho_\varepsilon$ is a probability measure too. Thus, using $\rho_\varepsilon$ as a competitor, we get
$$
\lambda (U_1(\bar{\rho}) - U_1(\rho_\varepsilon)) \leqslant \frac{W_2^2(\mu, \rho_\varepsilon) + W_2^2(\rho_\varepsilon, \nu)}{2 \tau} + \tau E(\rho_\varepsilon) - \frac{W_2^2(\mu, \bar{\rho}) + W_2^2(\bar{\rho}, \nu)}{2 \tau} - \tau E(\bar{\rho}).
$$
We estimate the r.h.s. by convexity (as $W_2^2$ and $F$ are convex) to see that 
$$
U_1(\bar{\rho}) - U_1(\rho_\varepsilon) \leqslant \frac{\varepsilon}{\lambda} \Bigg( \frac{W_2^2(\mu, \Leb) + W_2^2(\Leb, \nu)}{2 \tau} + \tau E(\Leb) - \frac{W_2^2(\mu, \bar{\rho}) + W_2^2(\bar{\rho}, \nu)}{2 \tau} - \tau E(\bar{\rho}) \Bigg).  
$$
Thus, there exists a constant $C$, independent of $\varepsilon$, such that $U_1(\bar{\rho}) - U_1(\rho_\varepsilon)  \leqslant C \varepsilon$. This can be easily seen to imply (see for instance the proof of \cite[Lemma 8.6]{SantambrogioOTAM}) that $\bar{\rho}$ is strictly positive a.e. 
\end{proof}

\noindent We can then write the first-order optimality conditions. 

\begin{prop}
\label{proposition_optimality_conditions}
The measure $\bar{\rho}$ (or more precisely its density w.r.t. $\Leb$) is Lipschitz and bounded away from $0$ and $\infty$. Moreover, let us denote by $\varphi_\mu$ and $\varphi_\nu$ the Kantorovitch potentials for the transport from $\bar{\rho}$ to respectively $\mu$ and $\nu$. Then the following identity holds a.e.:
\begin{equation}
\label{equation_optimality_conditions}
\frac{\nabla \varphi_\mu + \nabla \varphi_\nu}{\tau^2} + \left( f''(\bar{\rho}) + \frac{\lambda}{\bar{\rho}} \right) \nabla \bar{\rho}  + \nabla V = 0.
\end{equation}  
\end{prop}

\begin{proof}
Let $\tilde{\rho} \in \Prob(\Omega) \cap L^\infty(\Omega)$ and for $0 < \varepsilon < 1$ define $\rho_\varepsilon = (1 - \varepsilon) \bar{\rho} + \varepsilon \tilde{\rho}$. We use $\rho_\varepsilon$ as a competitor. We use Proposition \ref{proposition_first_variation_W2}: as $\bar{\rho} > 0$ a.e., the Kantorovitch potentials $\varphi_\mu$ and $\varphi_\nu$ for the transport from $\bar{\rho}$ to respectively $\mu$ and $\nu$ are unique and
\begin{equation*}
\lim_{\varepsilon \to 0} \frac{W_2^2(\mu, \bar{\rho}) - W_2^2(\mu, \rho_\varepsilon) + W_2^2(\bar{\rho}, \nu) - W_2^2(\rho_\varepsilon, \nu)}{2 \tau^2} = \int_\Omega \frac{\varphi_\mu + \varphi_\nu}{\tau} (\bar{\rho} - \tilde{\rho}).
\end{equation*} 
The term involving $V$ is straightforward to handle as it is linear. Hence, by optimality of $\bar{\rho}$ we get
\begin{equation}
\label{equation_derivation_optimality_auxiliary}
\int_\Omega \left( \frac{\varphi_\mu + \varphi_\nu}{\tau^2} + V \right) (\bar{\rho} - \tilde{\rho}) \leqslant \liminf_{\varepsilon \to 0} \frac{F(\rho_\varepsilon) + \lambda U_1(\rho_\varepsilon) - F(\bar{\rho}) - \lambda U_1(\bar{\rho})}{\varepsilon}.
\end{equation}   
\begin{equation*}
\frac{F(\rho_\varepsilon) + \lambda U_1(\rho_\varepsilon) - F(\bar{\rho}) - \lambda U_1(\bar{\rho})}{\varepsilon} = \int_\Omega \frac{f_\lambda[ (1 - \varepsilon) \bar{\rho} + \varepsilon \tilde{\rho} ] - f_\lambda[\bar{\rho}]}{\varepsilon}.
\end{equation*}
The integrand of the integral of the r.h.s. converges pointewisely, as $\varepsilon \to 0$, to $(f'(\bar{\rho}) + \lambda \ln \bar{\rho})(\tilde{\rho} - \bar{\rho})$. Moreover, as the function $f_\lambda$ is convex, we see that for $0 < \varepsilon < 1$, 
\begin{equation*}
\frac{f_\lambda[ (1 - \varepsilon) \tilde{\rho} + \varepsilon \bar{\rho} ] - f_\lambda[\bar{\rho}]}{\varepsilon} \leqslant f_\lambda(\tilde{\rho}) - f_\lambda(\bar{\rho}).
\end{equation*}
As $\tilde{\rho} \in L^\infty(\Omega)$ and $F(\bar{\rho}) + \lambda U_1(\bar{\rho}) < + \infty$, the r.h.s. of the equation is integrable on $\Omega$. Thus, by a reverse Fatou's lemma, 
\begin{equation*}
\limsup_{\varepsilon \to 0} \int_\Omega \frac{F(\rho_\varepsilon) + \lambda U_1(\rho_\varepsilon) - F(\bar{\rho}) - \lambda U_1(\bar{\rho})}{\varepsilon} \leqslant \int_\Omega \left( f'(\bar{\rho}) + \lambda \ln \bar{\rho} \right)(\tilde{\rho} - \bar{\rho}).
\end{equation*} 
Combing this equation with \eqref{equation_derivation_optimality_auxiliary}, we see that $\int_\Omega h \ddr (\tilde{\rho} -  \bar{\rho}) \geqslant 0$ with 
\begin{equation*}
h := \frac{\varphi_\mu + \varphi_\nu}{\tau^2} + f'(\bar{\rho}) + \lambda \ln \bar{\rho} + V. 
\end{equation*}
We know that $h$ is finite a.e., thus its essential infimum cannot be $+ \infty$. Moreover, starting from $\bar{\rho}f'(\bar{\rho}) \geqslant f(\bar{\rho}) - f(0)$, we see that $\int_\Omega h\bar{\rho} > - \infty$. Taking probability measures $\tilde{\rho}$ concentrated on sets where $h$ is close to its essential infimum, we see that the essential infimum of $h$ cannot be $+ \infty$ and that $h$ coincides with its essential infimum $\bar{\rho}$-a.e. As $\bar{\rho} > 0$ a.e., there exists $C$ such that we have a.e. on $\Omega$ 
\begin{equation}
\label{equation_derivation_optimality_integral}
f'(\bar{\rho}) + \lambda \ln \bar{\rho} = C - \frac{\varphi_\mu + \varphi_\nu}{\tau^2} - V. 
\end{equation}
As $f'$ is $C^1$ and increasing, it is easy to see that $f' + \lambda \ln$ is an homeomorphism of $\intervalleoo{0}{+ \infty}$ on $\intervalleoo{- \infty}{+ \infty}$ which is bilipschitz on compact sets. As the function $C - (\varphi_\mu + \varphi_\nu) / \tau^2 - V$ takes its values in a compact set and is Lispchitz, we see that $\bar{\rho}$ is bounded away from $0$ and $\infty$ and is Lipschitz. With all this regularity (recall that $f$ is assumed to be $C^2$ on $\intervalleoo{0}{+ \infty}$), we can take the gradient of \eqref{equation_derivation_optimality_integral} to obtain \eqref{equation_optimality_conditions}.
\end{proof}

\begin{theo}[Flow interchange inequality]
\label{theorem_flow_interchange}
For any $m \geqslant 1$, the following inequality holds: 
\begin{equation*}
\int_\Omega |\nabla \bar{\rho}|^2 f''(\bar{\rho}) \bar{\rho}^{m-1} + \int_\Omega (\nabla \bar{\rho} \cdot \nabla V) \bar{\rho}^{m-1} \leqslant \frac{U_m(\mu) + U_m(\nu) - 2 U_m(\bar{\rho})}{\tau^2}.
\end{equation*}
\end{theo}

\begin{proof}
We multiply pointewisely \eqref{equation_optimality_conditions} by $\bar{\rho}^{m-1} \nabla \bar{\rho}$ and integrate over $\Omega$. Dropping the entropic term, we easily get 
\begin{equation*}
\int_\Omega |\nabla \bar{\rho}|^2 f''(\bar{\rho}) \bar{\rho}^{m-1} + \int_\Omega (\nabla \bar{\rho} \cdot \nabla V) \bar{\rho}^{m-1} \leqslant - \frac{1}{\tau^2} \int_\Omega \left[ \nabla \bar{\rho} \cdot  (\nabla \varphi_\mu + \nabla \varphi_\nu ) \right] \bar{\rho}^{m-1}.  
\end{equation*}
To prove the flow interchange inequality, it is enough to show that 
\begin{equation*}
- \int_\Omega (\nabla \bar{\rho} \cdot \nabla \varphi_\mu) \bar{\rho}^{m-1} \leqslant U_m(\mu) - U_m(\bar{\rho}),
\end{equation*}
as a similar inequality will hold for the term involving $\varphi_\nu$. To this purpose, we denote by $\rho : \intervalleff{0}{1} \to \Prob(\Omega)$ the constant-speed geodesic joining $\bar{\rho}$ to $\mu$. We know that it is given by 
\begin{equation*}
\rho(t) = (\mathrm{Id} - t \nabla \varphi_\mu) \# \bar{\rho}.
\end{equation*} 
By geodesic convexity of $U_m$, the function $t \mapsto U_m(\rho(t))$ is convex. Hence, 
\begin{align*}
U_m(\mu) - U_m(\bar{\rho}) & = U_m(\rho(1)) - U_m(\rho(0)) \\
& \geqslant \limsup_{t \to 0} \frac{U_m(\rho(t)) - U_m(\rho(0))}{t} \\
& = \limsup_{t \to 0} \int_\Omega \frac{u_m(\rho(t)) - u_m(\bar{\rho})}{t} \\
& \geqslant \limsup_{t \to 0} \int_\Omega \frac{(\rho(t) - \bar{\rho}) u_m'(\bar{\rho})}{t} \\
& = \limsup_{t \to 0} \int_\Omega \frac{ u_m'(\bar{\rho}[ x - t \nabla \varphi_\mu(x)]) - u_m'(\bar{\rho}[x])}{t} \bar{\rho}(x) \ddr x,   
\end{align*}
where we also have used that $u_m$ is convex. It is clear that for a.e. $x \in \Omega$,
\begin{equation*}
\lim_{t \to 0} \frac{ u_m'(\bar{\rho}[ x - t \nabla \varphi_\mu(x)]) - u_m'(\bar{\rho}[x])}{t} = - \left[ (\nabla \bar{\rho} \cdot \nabla \varphi_\mu)  u_m''(\bar{\rho}) \right](x).  
\end{equation*}
Moreover, we have the uniform (in $t$) bound
\begin{equation*}
\left|  \frac{ u_m'(\bar{\rho}[ x - t \nabla \varphi_\mu(x)]) - u_m'(\bar{\rho}[x])}{t}  \right| \leqslant \| u_m''(\bar{\rho}) \|_\infty \| \nabla \bar{\rho} \|_\infty \| \nabla \varphi_\mu \|_\infty.
\end{equation*}
At this point, one can remember that $u_m''(x) = x^{m-2}$. Moreover, as $\bar{\rho}$ is bounded away from $0$ and $\infty$ and Lipschitz, the r.h.s. of the equation above is finite. Thus, by dominated convergence, 
\begin{equation*}
\limsup_{t \to 0}  \int_\Omega \frac{ u_m'(\bar{\rho}[ x - t \nabla \varphi_\mu(x)]) - u_m'(\bar{\rho}[x])}{t} \bar{\rho}(x) \ddr x = - \int_\Omega (\nabla \bar{\rho} \cdot \nabla \varphi_\mu) \bar{\rho}^{m-1}. \qedhere 
\end{equation*}
\end{proof}

\noindent From the result of Theorem \ref{theorem_flow_interchange} we need to deduce estimates on improved $L^m$ norms. To this aim, we treat in a slightly different way the cases of weak and strong congestion even if the result are similar. The main issue is to control the term involving $\nabla V$. 

\begin{crl}[Strong congestion case]
\label{corollary_flow_interchange_strong_congestion}
Suppose  that Assumption \ref{assumption_strong_congestion} holds. Then, for any $m \geqslant \alpha + 2$ one has 
\begin{equation*}
U_{\beta m}(\bar{\rho})^{1/ \beta} \leqslant C m^2 \left[ \frac{ U_m(\mu) + U_m(\nu) -2 U_m(\bar{\rho})}{\tau^2} + Cm^2 U_m(\bar{\rho}) \right],
\end{equation*}
where $C > 0$ depends only on $f, V$ and $\Omega$. 
\end{crl}

\begin{proof}
Let us start from the case of Assumption \ref{assumption_strong_congestion}. In this case, we recall that $C_f$ is the constant such that $f''(t) \geqslant C_f t^\alpha$ for any $t>0$. We transform the term involving $\nabla V$ in the following way: 
\begin{align*}
\int_\Omega (\nabla \bar{\rho} \cdot \nabla V) \bar{\rho}^{m-1} & = \int_\Omega ( \bar{\rho}^{\alpha / 2} \nabla \bar{\rho} ) \cdot (\bar{\rho}^{- \alpha/2} \nabla V) \bar{\rho}^{m-1}\\
& \geqslant - \frac{C_f}{2} \int_\Omega | \bar{\rho}^{\alpha / 2} \nabla \bar{\rho} |^2 \bar{\rho}^{m-1} - \frac{1}{2C_f} \int_\Omega |\bar{\rho}^{- \alpha/2} \nabla V|^2 \bar{\rho}^{m-1} \\
& = - \frac{C_f}{2} \int_\Omega | \nabla \bar{\rho} |^2 \bar{\rho}^{m-1 + \alpha} - \frac{1}{2 C_f} \int_\Omega |\nabla V|^2 \bar{\rho}^{m- 1 - \alpha} \\
& \geqslant  - \frac{C_f}{2} \int_\Omega | \nabla \bar{\rho} |^2 \bar{\rho}^{m-1 + \alpha} - \frac{\| \nabla V \|_\infty^2 m^2}{2C_f} U_m(\bar{\rho}).    
\end{align*}
For the last inequality, we have used the fact that
\begin{equation*}
\int_\Omega \bar{\rho}^{m-1- \alpha} 
\leqslant \left( \int_\Omega \bar{\rho}^m \right)^{(m-1 - \alpha)/m}  
\leqslant \int_\Omega \bar{\rho}^m
\leqslant m^2 U_m(\bar{\rho}),
\end{equation*}
which is valid because $1 \leqslant m - 1 - \alpha \leqslant m$ and $\Leb(\Omega) = 1$. Thus, using Theorem \ref{theorem_flow_interchange}, we get 
\begin{eqnarray*}
\frac{C_f}{2} \int_\Omega |\nabla \bar{\rho}|^2 \bar{\rho}^{m-1 + \alpha} &\leqslant&  \int_\Omega |\nabla \bar{\rho}|^2 f''(\bar{\rho}) \bar{\rho}^{m-1}-\frac{C_f}{2} \int_\Omega |\nabla \bar{\rho}|^2 \bar{\rho}^{m-1 + \alpha}\\
&\leqslant& \left[ \frac{ U_m(\mu) + U_m(\nu) -2 U_m(\bar{\rho})}{\tau^2} +  \frac{\| \nabla V \|_\infty^2}{2C_f} m^2 U_m(\bar{\rho}) \right].
\end{eqnarray*}
We are interested only in the large values taken by $\bar{\rho}$.  Let us introduce $\hat{\rho} := \max(1, \bar{\rho})$. This function is larger than $\bar{\rho}$ and $1$ and its gradient satisfies $|\nabla \hat{\rho}| = |\nabla \bar{\rho}| \1_{\bar{\rho} \geqslant 1} $. Thus, 
$$
\int_\Omega |\nabla \hat{\rho}^{m/2}|^2  = \frac{m^2}{4} \int_\Omega |\nabla \hat{\rho}|^2 \hat{\rho}^{m-2}  \leqslant \frac{m^2}{4} \int_\Omega |\nabla \hat{\rho}|^2 \hat{\rho}^{m-1 + \alpha}  \leqslant \frac{m^2}{4} \int_\Omega |\nabla \bar{\rho}|^2 \bar{\rho}^{m-1 + \alpha}. 
$$
(the last inequality is true since $\nabla\hat\rho=0$ on the points where $\hat\rho>\bar\rho$, and the first inequality is exactly the point where we exploit the fact $\hat\rho\geq 1$, which explains the use of $\hat\rho$ instead of $\bar\rho$). On the other hand, if we use the injection of $H^1(\Omega)$ into $L^{2 \beta}(\Omega)$ for the function $\hat{\rho}^{m/2}$, we get (with $C_\Omega$ a constant that depends only on $\Omega$), 
\begin{equation*}
\left( \int_\Omega \hat{\rho}^{m \beta } \right)^{1/ \beta} \leqslant C_\Omega \left( \int_\Omega |\nabla \hat{\rho}^{m/2}|^2 + \int_\Omega \hat{\rho}^m \right).
\end{equation*}
As $\bar{\rho}^{\beta m} \leqslant \hat{\rho}^{\beta m}$ and $\hat{\rho}^m \leqslant 1 + \bar{\rho}^m$, we see that 
\begin{align*}
\left( \int_\Omega \bar{\rho}^{m \beta } \right)^{1/ \beta} & \leqslant \left( \int_\Omega \hat{\rho}^{m \beta } \right)^{1/ \beta} \\
& \leqslant C_\Omega \left( \frac{m^2}{4} \int_\Omega | \nabla \bar{\rho} |^2 \bar{\rho}^{m-1 + \alpha} + \int_\Omega \bar{\rho}^m + 1 \right) \\
& \leqslant C_\Omega m^2 \left( \frac{1}{4} \int_\Omega | \nabla \bar{\rho} |^2 \bar{\rho}^{m-1 + \alpha} + 2 U_m(\bar{\rho}) \right) \\
& \leqslant C m^2\left[ \frac{ U_m(\mu) + U_m(\nu) -2 U_m(\bar{\rho})}{\tau^2} + Cm^2 U_m(\bar{\rho}) \right].
\end{align*}
Notice that to go from the second to the third line, we have used the fact that $1 \leqslant \int_\Omega \bar{\rho}^m \leqslant m^2 U_m(\bar{\rho})$. To conclude, it remains to notice that, as $m \beta \geqslant \beta > 1$, that we can control (uniformly in $m$) $U_{m \beta}(\bar{\rho})$ by $\int_\Omega \bar{\rho}^{m \beta}$. Indeed
\begin{equation*}
\left( \int_\Omega \bar{\rho}^{m \beta } \right)^{1/ \beta} \geqslant \frac{1}{(\beta (\beta - 1))^{1 / \beta}} U_{m \beta}(\bar{\rho})^{1/ \beta}.
\end{equation*}
Thus, up to a change in the constant $C$, we get the result we claimed. 
\end{proof}

\begin{crl}[Weak congestion case]
\label{corollary_flow_interchange_weak_congestion}
Suppose Assumption \ref{assumption_weak_congestion} and \ref{assumption_potential} both hold. Then, for any $m \geqslant 1$ such that $\beta (m+\alpha+1) \geqslant 1$ one has 
\begin{equation*}
\frac{ U_m(\mu) + U_m(\nu) -2 U_m(\bar{\rho})}{\tau^2} + Cm U_m(\bar{\rho}) \geqslant 0
\end{equation*}
and
\begin{equation*}
U_{\beta (m+1+\alpha)}(\bar{\rho})^{1/ \beta} \leqslant C m^2 \left[ \frac{ U_m(\mu) + U_m(\nu) -2 U_m(\bar{\rho})}{\tau^2} + Cm U_m(\bar{\rho}) \right] + C t_0^{m+1+\alpha},
\end{equation*}
where $C$ depends only on $f, V$ and $\Omega$. 
\end{crl}

\begin{proof} 
We use an integration by parts to treat the term involving $\nabla V$. Recall that $\mathbf{n}$ denotes the exterior normal to $\Omega$. 
\begin{align*}
\int_\Omega (\nabla \bar{\rho} \cdot \nabla V) \bar{\rho}^{m-1} & = \frac{1}{m} \int_\Omega \nabla (\bar{\rho}^m) \cdot \nabla V \\
& = \frac{1}{m} \int_{\partial \Omega} (\nabla V \cdot \mathbf{n}) \bar{\rho}^m - \frac{1}{m} \int_\Omega \Delta V \bar{\rho}^m \\
& \geqslant - \| \Delta V \|_\infty m U_m(\bar{\rho}), 
\end{align*}
where we have used the assumption $\nabla V \cdot \mathbf{n} \geqslant 0$ on $\partial \Omega$. Thus, using Theorem \ref{theorem_flow_interchange}, we get (recall that $f''(t) \geqslant C_f t^\alpha$ but only for $t \geqslant t_0$)
\begin{align}\label{rhobargeq1}
C_f \int_{ \{ \bar{\rho} \geqslant t_0 \}} |\nabla \bar{\rho}|^2 \bar{\rho}^{m-1 + \alpha} & \leqslant C_f \int_\Omega |\nabla \bar{\rho}|^2 \bar{\rho}^{m-1}f''(\bar\rho) \\
& \leqslant \left[ \frac{ U_m(\mu) + U_m(\nu) -2 U_m(\bar{\rho})}{\tau^2} +  \| \Delta V \|_\infty m U_m(\bar{\rho}) \right].
\end{align}
This gives us the first inequality of the corollary. In a similar manner to the strong congestion case, we introduce $\hat{\rho} := \max(t_0, \bar{\rho})$. This time we notice that 
\begin{equation*}
\int_\Omega |\nabla \hat{\rho}^{(m+1+\alpha)/2}|^2 \leqslant \frac{m^2}{4} \int_{ \{ \bar{\rho} \geqslant t_0 \}} |\nabla \bar{\rho}|^2 \bar{\rho}^{m-1 + \alpha}. 
\end{equation*}
Thus, if we use the injection of $H^1(\Omega)$ into $L^{2 \beta}(\Omega)$ with the function $\hat{\rho}^{(m+1+\alpha)/2}$, 
\begin{equation*}
\left( \int_\Omega \hat{\rho}^{\beta (m+1+ \alpha)} \right)^{1/\beta} \leqslant C_\Omega \left( \int_\Omega |\nabla \hat{\rho}^{(m+1+\alpha)/2}|^2 + \int_\Omega \hat{\rho}^{m+1+\alpha} \right). 
\end{equation*}
Then, we proceed as in the proof of the strong congestion case, but this time $m+1+\alpha \leqslant m$ and $\hat{\rho}^{m+1+\alpha} \leqslant \bar{\rho}^{m+1+\alpha} + t_0^{m+1+\alpha}$: 
\begin{align*}
\left( \int_\Omega \bar{\rho}^{\beta (m+1+ \alpha)} \right)^{1/\beta} & \leqslant \left( \int_\Omega \hat{\rho}^{\beta (m+1+ \alpha)} \right)^{1/\beta} \\
& \leqslant C_\Omega \left( \frac{m^2}{4} \int_{\{ \bar{\rho} \geqslant 1 \}} |\nabla \bar{\rho}|^2 \bar{\rho}^{m-1 + \alpha} + \int_\Omega \bar{\rho}^{m+1+ \alpha} + t_0^{m+1+\alpha} \right) \\
& \leqslant C_\Omega \left( \frac{m^2}{4} \int_{\{ \bar{\rho} \geqslant 1 \}} |\nabla \bar{\rho}|^2 \bar{\rho}^{m-1 + \alpha} + \int_\Omega \bar{\rho}^{m} + t_0^{m+1+\alpha} \right)  \\
& \leqslant C m^2 \left[ \frac{ U_m(\mu) + U_m(\nu) -2 U_m(\bar{\rho})}{\tau^2} + Cm U_m(\bar{\rho}) \right] + C t_0^{m+1+\alpha}. 
\end{align*}
Notice that if $t_0 \leqslant 1$, we can control $t_0^m$ by $m^2 U_m(\bar{\rho})$ (as we did in the strong congestion case), but in the general case this is not possible and we have to keep an explicit dependence in $t_0$. To conclude, we notice that, thanks to \eqref{equation_relation_beta_m0}, one has $\beta (m+1+\alpha) \geqslant m \geqslant m_0$ and thus 
\begin{equation*}
\left( \int_\Omega \bar{\rho}^{\beta (m+1+\alpha) } \right)^{1/ \beta}  \geqslant \frac{1}{(m_0 (m_0 - 1))^{1 / \beta}} U_{\beta(m+1+\alpha)}(\bar{\rho})^{1/ \beta}. \qedhere
\end{equation*}
\end{proof}

\noindent The last case is a combination of the previous two cases.

\begin{crl}[Strong congestion case-variant]
\label{corollary_flow_interchange_strong_congestion-variant}
Suppose that Assumption \ref{assumption_strong_congestion-variant} and \ref{assumption_potential} both hold. 
Then, for any $m \geqslant m_0$ one has
\begin{equation}
\label{ineq mtom2}
\frac{ U_m(\mu) + U_m(\nu) -2 U_m(\bar{\rho})}{\tau^2} + Cm U_m(\bar{\rho}) \geqslant 0
\end{equation}
and
\begin{equation}
\label{ineq mtom3}
U_{\beta m}(\bar{\rho})^{1/ \beta} \leqslant C m^2 \left[ \frac{ U_m(\mu) + U_m(\nu) -2 U_m(\bar{\rho})}{\tau^2} + Cm U_m(\bar{\rho}) \right] + C t_0^m,
\end{equation}
where $C$ depends only on $f, V$ and $\Omega$. 
\end{crl}

\begin{proof}
We begin with the same computations as in Corollary \ref{corollary_flow_interchange_weak_congestion}. We can obtain the same result as in \eqref{rhobargeq1}, but on the set $\{\bar \rho\geq t_0 \}$ we can use $\alpha\geq -1$ to write 
$$ \int_{ \{ \bar{\rho} \geqslant t_0 \}} |\nabla \bar{\rho}|^2 \bar{\rho}^{m} \leqslant C \left[ \frac{ U_m(\mu) + U_m(\nu) -2 U_m(\bar{\rho})}{\tau^2} +  \| \Delta V \|_\infty m U_m(\bar{\rho}) \right].$$
With $\hat{\rho} := \max(t_0, \bar{\rho})$ we get
\begin{equation*}
\int_\Omega |\nabla \hat{\rho}^{m/2}|^2 \leqslant C\frac{m^2}{4} \int_{ \{ \bar{\rho} \geqslant t_0 \}} |\nabla \bar{\rho}|^2 \bar{\rho}^{m},
\end{equation*}
and the conclusion comes from the same Sobolev injection, with the function $\hat{\rho}^{m/2}$, and similar computations as in the previous cases. 
\end{proof}

\begin{remark}\label{mtom2} For simplicity, inequality \eqref{ineq mtom2} and \eqref{ineq mtom3} will be used by replacing the term $mU_m(\bar\rho)$ with $m^2U_m(\bar\rho)$, so as to allow a unified presentation with the inequality obtained in Corollary \ref{corollary_flow_interchange_strong_congestion}. Notice also that Corollary \ref{corollary_flow_interchange_strong_congestion} is basically giving us the same inequality as \eqref{ineq mtom3}, as long as we set $t_0 = 0$. 
\end{remark}

%
%

\subsection{Boundary flow interchange}

In the case of Assumption \ref{assumption_temporal_bc}, we can derive some estimate right at the point $T$ ($k=N$). We will only sketch the proof, at it mimicks the proof of the interior case and these computations are well-known in the case of the applications to the JKO scheme. We know that, with $\bar{\rho} = \bar{\rho}^{N, \lambda}_T$ and $\mu := \bar{\rho}^{N, \lambda}_{T - \tau}$, the measure $\bar{\rho}$ is a minimizer (among all probability measures) of 
\begin{equation*}
\rho \mapsto \frac{W_2^2(\mu, \rho)}{2\tau} + G(\rho) + \lambda U_1(\rho) + \int_\Omega W \ddr \rho.
\end{equation*} 
Let us remark that it correspond to one step of the JKO scheme: it is in the context of such variational problems that the flow interchange was firstly used, see \cite{Matthes2009}. In any case, with these notation, we obtain:

\begin{prop}
\label{proposition_flow_interchange_boundary}
Suppose Assumption \ref{assumption_temporal_bc} holds. Then, for any $m \geqslant 1$ , 
\begin{equation*}
\frac{U_m(\mu) - U_m(\bar{\rho})}{\tau} \geqslant - (m-1) \| \Delta W \|_\infty U_m(\bar{\rho}). 
\end{equation*}
\end{prop}

\begin{proof}
Following the same strategy than in Lemma \ref{lemma_positivity_rho} and Proposition \ref{proposition_optimality_conditions}, we know that $\bar{\rho}$ is bounded away from $0$ and $\infty$, is a Lipschitz function, and that 
\begin{equation*}
\frac{\nabla \varphi_\mu}{\tau} + \left( g''(\bar{\rho}) + \frac{\lambda}{\bar{\rho}} \right) \nabla \bar{\rho} + \nabla W = 0 
\end{equation*} 
a.e. on $\Omega$, where $\varphi_\mu$ is the unique Kantorovitch potential for the transport from $\bar{\rho}$ to $\mu$. Thus, if we multiply by $\bar{\rho}^{m-1} \nabla \bar{\rho}$, we get, by the same estimation than in Theorem \ref{theorem_flow_interchange} (we drop both the entropic penalization and the congestion term), 
\begin{equation*}
\frac{U_m(\mu) - U_m(\bar{\rho})}{\tau} \geqslant \int_\Omega ( \nabla W \cdot \nabla \bar{\rho} ) \bar{\rho}^{m-1}.
\end{equation*} 
It remains to perform an integration by parts, using the sign of $\nabla W \cdot \mathbf{n}$ on $\partial \Omega$, to conclude that 
\begin{equation*}
\frac{U_m(\mu) - U_m(\bar{\rho})}{\tau}  \geqslant - \frac{1}{m} \int_\Omega \Delta W \bar{\rho}^m \geqslant - (m-1) \| \Delta W \|_\infty U_m(\bar{\rho}). \qedhere 
\end{equation*}

\end{proof}

\section{Moser-like iterations}

Corollaries \ref{corollary_flow_interchange_strong_congestion}, \ref{corollary_flow_interchange_weak_congestion} and \ref{corollary_flow_interchange_strong_congestion-variant} allow us to control the $L^{m \beta}$ or $L^{(m+1+\alpha)\beta}$ norm of $\bar{\rho}$ in terms of its $L^m$ norm. The strategy will consist in integrating w.r.t. to time and iterating such a control in order to get a bound on the $L^m(\intervalleff{T_1}{T_2} \times \Omega)$ norm of $\bar{\rho}^{N, \lambda}$ that does not depend on $\lambda$ and $N$ and to control how this bounds grows in $m$. 
For any $N \geqslant 1$ and any $0 < \lambda < 1$, recall that $\bar{\rho}^{N , \lambda}$ is a solution of the discrete problem \eqref{equation_discrete_problem}. 

\begin{defi} 
For any $m \geqslant 1$ and any $0 \leqslant T_1 \leqslant T_2 \leqslant T$, we define $L^m_{T_1, T_2}$ as 
\begin{equation*}
L^m_{T_1, T_2} := \liminf_{N \to + \infty, \lambda \to 0} \left( \sum_{T_1 \leqslant k \tau \leqslant T_2} \tau U_m(\bar{\rho}_{k \tau}^{N, \lambda}) \right)^{1/m}.
\end{equation*}
\end{defi} 

\noindent The quantity $L^m_{T_1, T_2}$ can be seen as a discrete counter part of (up to a factor $1/(m(m-1))^{1/ m}$) the $L^m$ norm of the restriction to $\intervalleff{T_1}{T_2}$ of the limit (whose existence will be proven in the next section) of $\bar{\rho}^{N, \lambda}$ when $N \to + \infty$ and $\lambda \to 0$.

\subsection{The strong congestion case}

First, we integrate w.r.t. time the estimate obtained in Corollary \ref{corollary_flow_interchange_strong_congestion}. 

\begin{prop}
\label{proposition_one_iteration_moser_like}
Suppose that either Assumption \ref{assumption_strong_congestion} holds or Assumptions \ref{assumption_strong_congestion-variant} and \ref{assumption_potential} both hold. Then there exists two constants $C_1$ and $C_2$ (depending on $f, V, T$ and $\Omega$) such that, for any $0 < \varepsilon \leqslant C_1/m$ and any $0 < T_1 < T_2 < T$ such that $\intervalleff{T_1 - \varepsilon}{T_2 + \varepsilon} \subset \intervalleoo{0}{T}$, and any $m \geqslant \alpha + 2$, 
\begin{equation*}
L^{\beta m}_{T_1, T_2} \leqslant \left[ C_2 \frac{m^3}{\varepsilon} \left( m^2 + \frac{1}{\varepsilon^2} \right) \right]^{1/m} \max \left( L^m_{T_1 - \varepsilon, T_2 + \varepsilon}, t_0 \right).
\end{equation*}
\end{prop}

\noindent As pointed out in Remark \ref{mtom2}, in the case where Assumption \ref{assumption_strong_congestion} holds, we set $t_0 = 0$. 

\begin{proof}
Let us recall that in Corollary \ref{corollary_flow_interchange_strong_congestion} and Corollary \ref{corollary_flow_interchange_strong_congestion-variant} with Remark \ref{mtom2}, we have proved (if we explicit the dependence in $N$ and $\lambda$) that for any $N \geqslant 1$, $\lambda > 0$ and any $k \in \{ 1, 2, \ldots, N-1 \}$, one has
\begin{equation}
\label{equation_iteration_start_flow_interchange}
U_{\beta m}(\bar{\rho}_{k \tau}^{N, \lambda})^{1/ \beta} \leqslant C m^2 \left[ \frac{ U_m(\bar{\rho}^{N, \lambda}_{(k-1) \tau}) + U_m(\bar{\rho}^{N, \lambda}_{(k+1) \tau}) -2 U_m(\bar{\rho}^{N, \lambda}_{k \tau})}{\tau^2} + Cm^2 U_m(\bar{\rho}_{k \tau}^{N, \lambda}) \right] + C t_0^m.
\end{equation}
Let us take $\chi : \intervalleff{0}{T} \to \intervalleff{0}{1}$ a positive $C^{1,1}$ cutoff function such that $\chi(t) = 1$ if $t \in \intervalleff{T_1 - \varepsilon /3}{T_2 + \varepsilon / 3}$ and $\chi(t) = 0$ if $t \notin \intervalleff{T_1 - 2 \varepsilon /3}{T_2 + 2 \varepsilon /3}$. Such a function $\chi$ can be chosen with $\| \chi'' \|_{\infty} \leqslant 54/ \varepsilon^2$. We multiply \eqref{equation_iteration_start_flow_interchange} by $\tau \chi(k \tau)$ and sum over $k \in \{ 1, 2, \ldots, N-1 \}$. After performing a discrete integration by parts, we are left with 
\begin{equation*}
\sum_{k=1}^{N-1} \tau \chi(k \tau) U_{m \beta}(\bar{\rho}^{N, \lambda}_{k \tau})^{1/ \beta} \leqslant C t_0^m + C m^2 \sum_{k=1}^{N-1} \tau U_m(\bar{\rho}^{N, \lambda}_{k \tau}) \left[ Cm^2 + \frac{\chi((k+1) \tau) + \chi((k-1) \tau) - 2 \chi(k \tau)}{\tau^2} \right]. 
\end{equation*}  
Given the bound on the second derivative of $\chi$, and if $\tau \leqslant \varepsilon / 3$, we get 
\begin{equation*}
\sum_{T_1 - \varepsilon/3 \leqslant k \tau \leqslant T_2 + \varepsilon / 3} \tau U_{m \beta}(\bar{\rho}^{N, \lambda}_{k \tau})^{1/ \beta} \leqslant C t_0^m + C m^2  \left( m^2 + \frac{1}{\varepsilon^2} \right) \sum_{T_1 - \varepsilon \leqslant k \tau \leqslant T_2 - \varepsilon} \tau U_m(\bar{\rho}^{N, \lambda}_{k \tau}). 
\end{equation*}
The l.h.s. is not exactly $\left( L^{m \beta}_{T_1, T_2} \right)^{1/m}$ as we would like to exchange the sum and the power $1/\beta$. Unfortunately, Jensen's inequality gives the inequality the other way around. To overcome this difficulty, we will use the fact that the function $k \mapsto U_{\beta m}(\bar{\rho}^{N, \lambda}_{k \tau})$ is almost a convex function of $k$. More precisely, we will use the "reverse Jensen inequality", whose proof is postponed in Appendix \ref{appendix}.

\begin{lm}
\label{lemma_reverse_Jensen}
Let $(u^\tau_k)_{k \in \Z}$ be a family of real sequences indexed by a parameter $\tau$. We assume that there exists $\omega \geqslant 0$ such that for any $k \in \Z$ and any $\tau$, one has $u^\tau_k > 0$ and 
\begin{equation}
\label{equation_almost_convex}
\frac{u^\tau_{k+1} + u^\tau_{k-1} - 2 u^\tau_{k}}{\tau^2} + \omega^2 u^\tau_k \geqslant 0. 
\end{equation}
Then, for any $T_1 < T_2$ and any $\eta < \pi / (8 \omega)$, there exists $\tau_0$ (which depends on $\omega$), such that, if $\tau \leqslant \tau_0$, then  
\begin{equation*}
\left( \sum_{T_1 \leqslant k \tau \leqslant T_2} \tau u_k^\tau \right)^{1/\beta} \leqslant C \frac{(\omega + 1)(T_2 - T_1 + 1)^{1+1/\beta}}{\eta} \ \sum_{T_1 - \eta  \leqslant k \tau \leqslant T_2 + \eta} \tau (u_k^\tau)^{1/ \beta},
\end{equation*}  
where $C$ is a universal constant.  
\end{lm}

\noindent To use this lemma, we observe that $u^\tau_k := U_{m \beta}(\bar{\rho}^{N, \lambda}_{k \tau})$ satisfies \eqref{equation_almost_convex} with $\omega^2 = Cm^2$ (thanks again to Corollary \ref{corollary_flow_interchange_strong_congestion} and Remark \ref{mtom2}). Thus, if we take $C_1$ small enough, we have $\varepsilon /3 < \pi / (8 \omega)$ as soon as $\varepsilon \leqslant C_1 / m$. If $\tau$ is small enough, we can exchange the sum and the power $1/\beta$ to get 
\begin{align*}
\left( \sum_{T_1 \leqslant k \tau \leqslant T_2} \tau U_{m \beta}(\bar{\rho}^{N, \lambda}_{k \tau}) \right)^{1/ \beta} & \leqslant C \frac{(C m + 1) (T+1)^{1+1/\beta}}{\varepsilon} \sum_{T_1 - \varepsilon/3 \leqslant k \tau \leqslant T_2 + \varepsilon / 3} \tau U_{m \beta}(\bar{\rho}^{N, \lambda}_{k \tau})^{1/ \beta} \\
& \leqslant C_2 \frac{m^3}{\varepsilon} \left( m^2 + \frac{1}{\varepsilon^2} \right) \left( t_0 + \sum_{T_1 - \varepsilon \leqslant k \tau \leqslant T_2 + \varepsilon} \tau U_m(\bar{\rho}^{N, \lambda}_{k \tau}) \right).
\end{align*}
Notice that we have put the constant $C_2 m^3 \varepsilon^{-1} (m^2 + \varepsilon^{-2})$ also in factor of $t_0^m$, as it is anyway larger than $1$ as soon as $\varepsilon$ is small enough. Then we take the power $1/m$ on both sides, use the identity $(a+b)^{1/m} \leqslant C \max (a^{1/m}, b^{1/m})$ and send $N \to + \infty$ and $\lambda \to 0$ to get the result.
\end{proof}

\noindent In other words, on a slightly larger time interval, the $L^{\beta m}$ norm is control by the $L^m$ norm. We just have to iterate this inequality. 

\begin{prop}
\label{proposition_conclusion_iteration_flow_interchange}
Suppose that either Assumption \ref{assumption_strong_congestion} holds or Assumptions \ref{assumption_strong_congestion-variant} and \ref{assumption_potential} both hold. For any $0 < T_1 < T_2 < T$, there exists $C$ (that depends on $T_1, T_2, T, f, V$ and $\Omega$) such that 
\begin{equation*}
\limsup_{m \to + \infty} L^m_{T_1, T_2}  \leqslant C \max \left( L^{\alpha + 2}_{0,T}, t_0 \right).
\end{equation*}
\end{prop}

\begin{proof}
Let $\varepsilon_0 > 0$ be small enough such that $0 < T_1 - \varepsilon_0 \beta / (\beta - 1) \leqslant T_2 + \varepsilon_0 \beta / (\beta - 1) < 1 $ and $\varepsilon_0 \leqslant C_1 / (\alpha + 2)$ (where $C_1$ is the constant defined in Proposition \ref{proposition_one_iteration_moser_like}). For any $n \in \N$, let us define
\begin{equation*}
T^n_1 := T_1 - \sum_{k=n}^{+ \infty} \frac{\varepsilon_0}{\beta^n} \ \text{ and } \ T^n_2 := T_2 + \sum_{k=n}^{+ \infty} \frac{\varepsilon_0}{\beta^n},
\end{equation*} 
and set $m_n := (\alpha + 2) \beta^n$. Using Proposition \ref{proposition_one_iteration_moser_like}, as we have $|T^{n+1}_i - T^n_i| = \varepsilon_0 \beta^{-n} \leqslant C_1 / m_n$ for $i \in \{ 1,2 \}$, we can say that, with $l_n := \max \left( L^{m_n}_{T^n_1, T^n_2}, t_0 \right)$ 
\begin{align*}
l_{n+1} & \leqslant \left[ \max\left\{1,C_2 \frac{m_n^3}{\varepsilon_0 \beta^{-n}} \left( m_n^2 + \frac{1}{(\varepsilon_0 \beta ^{-n})^2} \right) \right\}\right]^{1/m_n} l_n  \\
& \leqslant \left[ C \beta^{6n} \right]^{\beta^{-n}/(\alpha + 2)} l_n. 
\end{align*}    
One can easily check, as $\beta > 1$, that 
\begin{equation*}
\prod_{n=0}^{+ \infty} \left[ C \beta^{6n} \right]^{\beta^{-n}/(\alpha + 2)}  < + \infty,
\end{equation*}
thus we get that 
$$
\sup_{n \in \N} L^{m_n}_{T_1, T_2}  \leqslant \sup_{n \in \N} L^{m_n}_{T^n_1, T^n_2} \leqslant \sup_{n \in \N} l_n \leqslant C l_0 = C \max \left( L^{m_0}_{T_1^0, T_2^0}, t_0 \right) \leqslant C \max \left( L^{\alpha + 2}_{0,T}, t_0 \right). 
$$
To conclude, we notice that, if $m > 1$ and $m_n \geqslant m$, one has (using Jensen's inequality)
\begin{equation*}
L^m_{T_1, T_2} \leqslant \frac{(m_n(m_n -1))^{1/m_n}}{(m(m-1))^{1/m}} L^{m_n}_{T_1, T_2},
\end{equation*}
thus sending $m \to + \infty$ (hence $n \to + \infty$) we conclude that 
\begin{equation}
\label{equation_control_Lm_Lmn_aux}
\limsup_{m \to + \infty} L^m_{T_1, T_2} \leqslant \sup_{n \in \N} L^{m_n}_{T_1, T_2}. \qedhere 
\end{equation}
\end{proof}

\noindent As we will see later, the fact that $L^{\alpha + 2}_{0,T}$ is finite is a consequence of the fact that the solution $\bar{\rho}$ of the continuous problem \eqref{equation_continuous problem} satisfies $\int_0^T F(\bar{\rho}_t) \ddr t < + \infty$.

\subsection{Estimates up to the final time}

In this subsection, still supposing that either Assumption \ref{assumption_strong_congestion} holds or Assumptions \ref{assumption_strong_congestion-variant} and \ref{assumption_potential} both hold, we exploit Assumption \ref{assumption_temporal_bc} to extend the $L^\infty$ bound up to the final time $t=T$. We will prove a result similar to Proposition \ref{proposition_one_iteration_moser_like}, but this time up to the boundary. 

\begin{prop}
\label{proposition_one_iteration_boundary}
Suppose that either Assumption \ref{assumption_strong_congestion} holds or Assumptions \ref{assumption_strong_congestion-variant} and \ref{assumption_potential} both hold, and that Assumption \ref{assumption_temporal_bc} also holds. Then there exists two constants $C_1$ and $C_2$ (depending on $f, T, V, g, W$ and $\Omega$) such that for any $0 < \varepsilon < C_1 / m$ and any $0 < T_1 <T$ with $0 < T_1 - \varepsilon$, then  for any $m \geqslant \alpha + 2$, 
\begin{equation*}
L^{\beta m}_{T_1, T} \leqslant \left[ C_2 \frac{m^3}{\varepsilon} \left( \frac{m}{\varepsilon} + m^2 + \frac{1}{\varepsilon^2} \right) \right]^{1/m} \max \left( L^m_{T_1 - \varepsilon, T}, t_0 \right).
\end{equation*}
\end{prop} 

\noindent Again, we recall (Remark \ref{mtom2}) that if we are under Assumption \ref{assumption_strong_congestion}, we take $t_0 = 0$. 

\begin{proof}
Let us recall that equation \eqref{equation_iteration_start_flow_interchange} holds for any $N \geqslant 1$, $\lambda > 0$ and $k \in \{ 1,2, \ldots, N-1 \}$. We take $\chi : \intervalleff{0}{T} \to \intervalleff{0}{1}$ a positive $C^{1,1}$ cutoff function such that $\chi(t) = 1$ if $t \in \intervalleff{T_1 - \varepsilon /3}{T}$ and $\chi(t) = 0$ if $t \in \intervalleff{0}{T_1 - 2 \varepsilon /3}$. Such a function $\chi$ can be chosen with $\| \chi'' \|_{\infty} \leqslant 54/ \varepsilon^2$. We multiply \eqref{equation_iteration_start_flow_interchange} by $\tau \chi(k \tau)$ and sum over $k \in \{ 1, 2, \ldots, N-1 \}$. After performing a discrete integration by parts, we are left with (now a boundary term is appearing): 
\begin{multline*}
\sum_{k=1}^{N-1} \tau \chi(k \tau) U_{m \beta}(\bar{\rho}^{N, \lambda}_{k \tau})^{1/ \beta}\! \leqslant \\ C m^2 \Bigg( \frac{U_m(\bar{\rho}^{N, \lambda}_{T})\! -\! U_m(\bar{\rho}^{N, \lambda}_{T - \tau} )}{\tau} \chi(T)   +\! \sum_{k=1}^{N-1} \!\tau U_m(\bar{\rho}^{N, \lambda}_{k \tau})\! \left[\! Cm^2 + \frac{\chi(k\tau+ \tau) \!+\! \chi(k\tau-\tau) \!-\! 2 \chi(k \tau)}{\tau^2}\! \right] \! \Bigg) + C t_0^m. 
\end{multline*} 
With the help of Proposition \ref{proposition_flow_interchange_boundary} and Corollary \ref{corollary_flow_interchange_strong_congestion} or Corollary \ref{corollary_flow_interchange_strong_congestion-variant}, and as $\chi(T) = 1$, we are able to write (provided that $\tau \leqslant \varepsilon / 3$)
\begin{equation*}
\sum_{T_1 - \varepsilon /3 \leqslant k \tau \leqslant T} \tau U_{m \beta}(\bar{\rho}^{N, \lambda}_{k \tau})^{1/ \beta} \leqslant C m^2 \left( m U_m(\bar{\rho}^{N, \lambda}_T)  + \left[ m^2 + \frac{1}{\varepsilon^2} \right] \sum_{T_1 - \varepsilon \leqslant k \tau \leqslant T} \tau U_m(\bar{\rho}^{N, \lambda}_{k \tau})   \right) + C t_0^m. 
\end{equation*} 
To transform the boundary term $U_m(\bar{\rho}^{N, \lambda}_T)$ into an integral term, we use the following lemma, whose proof is also postponed in Appendix \ref{appendix}. 

\begin{lm}
\label{lemma_reverse_Jensen_Neumann}
Let $(u^\tau_k)_{k \in \Z}$ be a family of real sequences indexed by a parameter $\tau$. We assume that there exists $\omega \geqslant 0$ such that for any $k \in \Z$ and any $\tau$, one has $u^\tau_k > 0$ and \eqref{equation_almost_convex}. We also assume that there exists $b \geqslant 0$ such that for some $N \in \Z$,
\begin{equation*}
\frac{u_{N}^\tau - u^\tau_{N-1}}{\tau} \leqslant  b u_N\tau.  
\end{equation*} 
Then, there exists $C_1$ and $C_2$ universal constants and $\tau_0$ (which depends on $\omega$ and $b$), such that for any $\eta \leqslant \min\{ \pi / (32 \omega), \pi / (32b) \}$ and any $\tau \leqslant \tau_0$, then 
\begin{equation}
\label{equation_reverse_Jensen_boundary_control}
u^\tau_N \leqslant \frac{C_1}{\eta} \sum_{kN - \eta \leqslant k \tau \leqslant kN} \tau u^\tau_k,
\end{equation}
and for any $T_1 < N \tau$, 
\begin{equation}
\label{equation_reverse_Jensen_boundary_inverse}
\left( \sum_{T_1 \leqslant k \tau \leqslant N \tau} \tau u_k^\tau \right)^{1/\beta} \leqslant C_2 \frac{(\omega  + 1)(T - T_1 + 1)^{1+1/\beta}}{\eta} \ \sum_{T_1 - \eta  \leqslant k \tau \leqslant N \tau} \tau (u_k^\tau)^{1/ \beta}.
\end{equation}
\end{lm}

\noindent We are in the case where this lemma can be applied because of Corollary \ref{corollary_flow_interchange_strong_congestion} or Corollary \ref{corollary_flow_interchange_strong_congestion-variant} and Proposition \ref{proposition_flow_interchange_boundary} with $u^\tau_k = U_m(\bar{\rho}^{N, \lambda}_{k \tau})$, $\omega = C m$ and $b = C m$. Thus, if $\varepsilon < C / m$, we can guarantee that we can use equation \eqref{equation_reverse_Jensen_boundary_control} of Lemma \ref{lemma_reverse_Jensen_Neumann} (with $\varepsilon = \eta$), thus
\begin{equation*}
\sum_{T_1 - \varepsilon /3 \leqslant k \tau \leqslant T} \tau U_{m \beta}(\bar{\rho}^{N, \lambda}_{k \tau})^{1/ \beta} \leqslant C t_0^m + C m^2 \left[ \frac{m}{\varepsilon} + m^2 + \frac{1}{\varepsilon^2} \right] \sum_{T_1 - \varepsilon \leqslant k \tau \leqslant T} \tau U_m(\bar{\rho}^{N, \lambda}_{k \tau}). 
\end{equation*}
Then we use equation \eqref{equation_reverse_Jensen_boundary_inverse} of Lemma \ref{lemma_reverse_Jensen_Neumann} (but this time with $u^\tau_k = U_{\beta m}(\bar{\rho}^{N, \lambda}_{k \tau})$) to exchange the sum and the power $1/\beta$ on the l.h.s. to conclude that 
\begin{equation*}
\left( \sum_{T_1  \leqslant k \tau \leqslant T} \tau U_{m \beta}(\bar{\rho}^{N, \lambda}_{k \tau}) \right)^{1 / \beta} \leqslant C \frac{m^3}{\varepsilon} \left[ \frac{m}{\varepsilon} + m^2 + \frac{1}{\varepsilon^2} \right] \left( t_0^m + \sum_{T_1 - \varepsilon \leqslant k \tau \leqslant T} \tau U_m(\bar{\rho}^{N, \lambda}_{k \tau}) \right).
\end{equation*}
Again, we have put $m^3 \varepsilon^{-1} (m \varepsilon^{-1} + m^2 + \varepsilon^{-2})$ in factor of $t_0^m$, which is legit because this factor is larger than $1$ for $\varepsilon$ small enough. Taking the power $1/m$ on each side, using the identity $(a+b)^{1/m} \leqslant C \max (a^{1/m}, b^{1/m})$, and letting $N \to + \infty$ and $\lambda \to 0$, we get the result.
\end{proof}

It is then very easy to iterate this result, which looks exactly like Proposition \ref{proposition_one_iteration_moser_like}. Thus, the proof of the following proposition, which is exactly the same as Proposition \ref{proposition_conclusion_iteration_flow_interchange}, is left to the reader. 

\begin{prop}
\label{proposition_conclusion_iteration_flow_interchange_boundary}
Suppose that either Assumption \ref{assumption_strong_congestion} holds or Assumptions \ref{assumption_strong_congestion-variant} and \ref{assumption_potential} both hold, and that Assumption \ref{assumption_temporal_bc} also holds. Then, for any $0 < T_1 < T$, there exists $C$ (that depends on $T_1, T, f, V$ and $\Omega$) such that 
\begin{equation*}
\limsup_{m \to + \infty} L^m_{T_1, T}  \leqslant C \max \left( L^{\alpha + 2}_{0,T}, t_0 \right).
\end{equation*}
\end{prop}

\subsection{The weak congestion case}

The scheme is very similar in the weak congestion case, even though the iteration is not as direct as in the strong congestion case. Moreover, we will directly prove an $L^\infty$ bound up to $t=T$, because, as we will see, Assumption \ref{assumption_temporal_bc} will be needed anyway to initialize the iterative process. The proofs will be less detailed in this case: the reading on the two previous subsections is advised to understand this one. 

\begin{prop}
\label{proposition_one_iteration_moser_like_weak}
Suppose Assumptions \ref{assumption_weak_congestion}, \ref{assumption_potential} and \ref{assumption_temporal_bc} hold. Then there exist constants $C_1$ and $C_2$ (depending on $f, V, T$ and $\Omega$) such that, for any $0 < \varepsilon \leqslant C_1/ m$ and any $0 < T_1 < T$ such that $0 < T_1 - \varepsilon$, then for any $m \geqslant m_0$, 
\begin{equation*}
L^{\beta (m+1+\alpha)}_{T_1, T} \leqslant \left[ C_2 \frac{m^{5/2}}{\varepsilon} \left(\frac{m}{\varepsilon} +  m + \frac{1}{\varepsilon^2} \right) \right]^{1/(m+1+\alpha)} \max \left[ \left( L^m_{T_1 - \varepsilon, T} \right)^{m/(m+1+\alpha)}, t_0 \right].
\end{equation*}
\end{prop}

\begin{proof}
The proof starts the same way: starting from Corollary \ref{corollary_flow_interchange_weak_congestion}, we write 
\begin{equation}
\label{equation_iteration_start_flow_interchange_weak}
U_{\beta (m+1+\alpha)}(\bar{\rho}_{k \tau}^{N, \lambda})^{1/ \beta} \leqslant C m^2 \left[ \frac{ U_m(\bar{\rho}^{N, \lambda}_{(k-1) \tau}) + U_m(\bar{\rho}^{N, \lambda}_{(k+1) \tau}) -2 U_m(\bar{\rho}^{N, \lambda}_{k \tau})}{\tau^2} + Cm U_m(\bar{\rho}_{k \tau}^{N, \lambda}) \right] + C t_0^{m+1+\alpha}.
\end{equation}
Because of Assumption \ref{assumption_temporal_bc}, we can also write, tanks to Proposition \ref{proposition_flow_interchange_boundary}, that 
\begin{equation*}
\frac{U_m(\bar{\rho}^{N, \lambda}_{T - \tau}) - U_m(\bar{\rho}^{N, \lambda}_T)}{\tau} \geqslant - (m-1) \| \Delta W \|_\infty U_m(\bar{\rho}^{N, \lambda}_{T}).
\end{equation*}
We use the same cutoff function $\chi$ that in the proof of Proposition \ref{proposition_one_iteration_boundary}. We multiply \eqref{equation_iteration_start_flow_interchange_weak} by $\tau \chi(k \tau)$, perform a discrete integration by parts and end up with 
\begin{align*}
\sum_{T_1 - \varepsilon/3 \leqslant k \tau \leqslant T} \tau U_{\beta(m+1+\alpha)}(\bar{\rho}^{N, \lambda}_{k \tau})^{1/ \beta} & 
\leqslant C m^2 \Bigg( \frac{U_m(\bar{\rho}^{N, \lambda}_{T}) - U_m(\bar{\rho}^{N, \lambda}_{T - \tau} )}{\tau}   + \left[ m + \frac{1}{\varepsilon^2} \right] \sum_{T_1 - \varepsilon \leqslant k \tau \leqslant T} \tau U_m(\bar{\rho}^{N, \lambda}_{k \tau})   \Bigg) + C t_0^{m+1+\alpha} \\
& \leqslant C m^2 \Bigg( m U_m(\bar{\rho}^{N, \lambda}_T)   + \left[ m + \frac{1}{\varepsilon^2} \right] \sum_{T_1 - \varepsilon \leqslant k \tau \leqslant T} \tau U_m(\bar{\rho}^{N, \lambda}_{k \tau})   \Bigg) + C t_0^{m+1+\alpha}.
\end{align*}
We also use Lemma \ref{lemma_reverse_Jensen_Neumann} but this time with $\omega^2 = Cm$ (this is Corollary \ref{corollary_flow_interchange_weak_congestion}) and $b = C m$. The frequency $\omega^2$ is smaller than in the strong congestion case (where it was of the order $m^2$) because we have made stronger assumptions on the potential $V$, though this is not important as we only use the fact that $\omega$ grows not faster than a polynomial of $m$. With this lemma we can both transform the boundary term into an integral term and exchange the sum and the power $1/ \beta$: there exists $C_1$ such that if $0 < \varepsilon  \leqslant C_1/m$ and if $\tau$ is small enough, 
\begin{align*}
\left( \sum_{T_1 \leqslant k \tau \leqslant T} \tau U_{\beta(m+1+\alpha)}(\bar{\rho}^{N, \lambda}_{k \tau}) \right)^{1/ \beta} & \leqslant C \frac{( \sqrt{m} + 1) T^{1+1/\beta}}{\varepsilon} \sum_{T_1 - \varepsilon/3 \leqslant k \tau \leqslant T} \tau U_{\beta (m+1+\alpha)}(\bar{\rho}^{N, \lambda}_{k \tau})^{1/ \beta} \\
& \leqslant C \frac{m^{5/2}}{\varepsilon} \left( \frac{m}{\varepsilon} +  m + \frac{1}{\varepsilon^2} \right) \left( t_0^{m+1+\alpha} + \sum_{T_1 - \varepsilon \leqslant k \tau \leqslant T} \tau U_m(\bar{\rho}^{N, \lambda}_{k \tau}) \right).
\end{align*}
We take the power $1/(m+1+\alpha)$ on both sides, use the fact that $(a+b)^{1/(m+1+\alpha)} \leqslant C \max (a^{1/(m+1+\alpha)}, b^{1/(m+1+\alpha)})$, and let $N \to + \infty, \lambda \to 0$ to get the result.
\end{proof}

\noindent We proceed the same way by iterating the inequality, even though this expressions are slightly more complicated. Let us underline that the condition \eqref{equation_relation_beta_m0} on $\beta$ is precisely the one that ensures that $\beta(m+1+\alpha) > m$ as soon as $m \geqslant m_0$: it is only thanks to this condition that the iteration of Proposition \ref{proposition_one_iteration_moser_like_weak} will give useful information. 

\begin{prop}
\label{proposition_conclusion_iteration_flow_interchange_weak}
Suppose Assumptions \ref{assumption_weak_congestion}, \ref{assumption_potential} and \ref{assumption_temporal_bc} hold. Then, there exists $\gamma < + \infty$ such that, for any $0 < T_1 < T$, there exists $C$ (that depends on $T_1, T, f, V$ and $\Omega$) such that 
\begin{equation*}
\limsup_{m \to + \infty} L^m_{T_1, T} \leqslant C  \left( \max \left[  L^{m_0}_{0,T} , t_0 \right] \right)^\gamma.
\end{equation*}
\end{prop}

\begin{proof}
As we know, thanks to our normalization choices, that $L^{m_0}_{0,T} \geqslant 1$, it is not restrictive that assume that $t_0 \geqslant 1$ (indeed, if this is not the case, Assumption \ref{assumption_weak_congestion} is still valid with $t_0 = 1$ and the content of Proposition \ref{proposition_conclusion_iteration_flow_interchange_weak} does not change). 

Once we have chosen $\varepsilon_0 < \beta T_1 / (\beta -1)$, we define $T^n_1$ by the same formula as in the proof of Proposition \ref{proposition_conclusion_iteration_flow_interchange}. We also define $m_n$ by recurrence: for any $n \in \N$, we take $m_{n+1} = \beta (m_n + 1 + \alpha)$. Thus, we have the explicit expression 
\begin{equation*}
m_n = \left( m_0 + (\alpha + 1) \frac{\beta}{\beta -1} \right) \beta^n - (\alpha + 1) \frac{\beta}{\beta -1}.
\end{equation*}
In particular, $(m_n)_{n \in \N}$ diverges exponentially fast to $+ \infty$ as $n \to + \infty$. Using Proposition \ref{proposition_one_iteration_moser_like_weak} and as $t_0 \geqslant 1$, we get 
\begin{align*}
L^{m_{n+1}}_{T^{n+1}_1, T} & \leqslant \left[ C_2 \frac{m_n^{5/2}}{\varepsilon_0 \beta^{-n}} \left( \frac{m_n}{\varepsilon_0 \beta^{-n}} + m_n + \frac{1}{(\varepsilon_0 \beta ^{-n})^2} \right) \right]^{1/(m_n+1+\alpha)} \max \left[ \left( L^{m_n}_{T^n_1, T} \right)^{m_n/(m_n + 1 + \alpha)}, t_0 \right]  \\
& \leqslant \left[ C \beta^{11n/2} \right]^{1/(m_n+\alpha+1)} \max \left( \left[  L^{m_n}_{T^n_1, T}, t_0 \right] \right)^{m_n/(m_n + \alpha +1)}. 
\end{align*}  
Denoting by $l_n := \ln \left( \max \left[ L^{m_n}_{T^n_1, T}, t_0 \right] \right)$, we see that 
\begin{equation*}
l_{n+1} \leqslant C_3 \frac{11n}{2(m_n + \alpha +1)} + \frac{C_4}{m_n+\alpha +1} + \frac{m_n}{m_n + \alpha +1} l_n.  
\end{equation*}
Given the exponential asymptotic growth of $(m_n)_{n \in \N}$, we leave it to the reader to check that is enough to conclude that $\limsup_{n \to + \infty} l_n \leqslant \gamma l_0 + C_5$ for some $\gamma < + \infty$ and $C_4 < + \infty$. Taking the exponential gives
\begin{equation*}
\limsup_{n \in \N} L^{m_n}_{T^n_1, T} \leqslant C  \left( \max \left[  L^{m_0}_{0,T} , t_0 \right] \right)^\gamma.
\end{equation*}   
To conclude, we use again \eqref{equation_control_Lm_Lmn_aux}, which is valid independently of Assumption \ref{assumption_strong_congestion} or Assumption \ref{assumption_weak_congestion}. 
\end{proof}

However, in the weak congestion case, the fact that $L_{0,T}^{m_0} < +\infty$ will require a little bit more of work and relies on the particular form of the boundary conditions.   

\begin{prop}
\label{proposition_initialisation_weak_congestion}
Suppose Assumptions \ref{assumption_weak_congestion}, \ref{assumption_potential} and \ref{assumption_temporal_bc} hold.  Then there exists $T_\text{max}$ (which depends on $f, V, \Psi$ and $\Omega$) such that, if $T \leqslant T_\text{max}$,  
\begin{equation*}
L_{0,T}^{m_0} < + \infty.
\end{equation*}
\end{prop}

\begin{proof}
Again we will use the almost convexity of $U_{m_0}(\bar{\rho}^{N, \lambda})$. Indeed, we rely on the following lemma, which has the same flavor as the "reverse Jensen inequality" and whose proof is postponed in Appendix \ref{appendix}. 

\begin{lm}
\label{lemma_reverse_Jensen_boundary}
Let $a >0$, $b \geqslant 0$ and $\omega \geqslant 0$ and set $T_\mathrm{max} = \min \{ \pi /(32 \omega), \pi / (32b) \}$. Then there exist some constants $C<+ \infty$ and $\tau_0 >0$ (all depending on $a, b$ and $\omega$) such that for any $T \leqslant T_\text{max}$, any $N > 1/ \tau_0$ ($\tau := T/N$) and for any sequence $(u_k^\tau)_k \in \Z$ of strictly positive numbers satisfying \eqref{equation_almost_convex} for $k \in \{ 1,2, \ldots, N-1 \}$, and such that $u_0^\tau = a$ and 
\begin{equation*}
\frac{u_{N-1}^\tau - u^\tau_N}{\tau} \geqslant -b u_N^\tau,
\end{equation*}
one has $u^\tau_k \leqslant C$ for any $k \in \{ 1,2, \ldots, N \}$. \end{lm}

\noindent We use this lemma with $u_k^\tau = U_{m_0}(\bar{\rho}^{N, \lambda}_{k \tau})$. Equation \eqref{equation_almost_convex} is satisfied with $\omega^2 = C m_0$ (Corollary \ref{corollary_flow_interchange_weak_congestion}); one can take 
$$
a = U_{m_0} (\bar{\rho}^{N, \lambda}_0) = U_{m_0} (\overline{\rho_0})= \frac{1}{m_0(m_0-1)} \int_\Omega \overline{\rho_0}^{m_0}; 
$$
and we take $b = (m_0 -1) \| \Delta W \|_\infty$ (cf. Proposition \ref{proposition_flow_interchange_boundary}). Thus, one can conclude that if $T \leqslant T_\text{max}$, then $U_{m_0}(\bar{\rho}^{N, \lambda}_{k \tau})$ is bounded independently on $N$. This is enough to conclude that $L_{0,T}^{m_0}$ is finite.     
\end{proof}

\section{Limit of the discrete problems}

In this section, we will see that the solutions $\bar{\rho}^{N, \lambda}$ of the discrete problems \eqref{equation_discrete_problem} converge to the solution $\bar{\rho}$ of the continuous one \eqref{equation_continuous problem} when $N \to + \infty$ and $\lambda \to 0$. Then, using the results of the previous sections, we will be able to show the $L^\infty$ bound on $\bar{\rho}$.  

\subsection{Building discrete curves from continuous one}

In our construction we will need to work with curves with finite entropy. This is easy under Assumption \ref{assumption_strong_congestion} of \ref{assumption_strong_congestion-variant}, but requires an approximation in the case of Assumption \ref{assumption_weak_congestion}. Hence, we will show that in this case we can approximate curves in $\Gamma$ by curves in $\Gamma$ with finite entropy. In order to do this, we will use the heat flow, whose definition and some useful properties are recalled below. For any $s \geqslant 0$ and any $\mu \in \Prob(\Omega)$, let us define $\Phi_s \mu := u(s)$, where $u$ is the solution of the Cauchy problem 
\begin{equation*}
\begin{cases}
\partial_t u= \Delta u & \text{in } \intervalleoo{0}{+ \infty} \times \mathring{\Omega} \\
\nabla u  \cdot \mathbf{n} = 0 & \text{on } \intervalleoo{0}{+ \infty} \times \partial \Omega \\
\dst{\lim_{t \to 0} u(t)} = \mu & \text{in } \Prob(\Omega)
\end{cases}.
\end{equation*}
In the equation above, $\mathbf{n}$ stands for the normal vector to the boundary $\partial \Omega$. Provided that the boundary $\partial \Omega$ of $\Omega$ is smooth, it is well known (see for instance \cite[Section 7]{Arendt2002} and \cite{Pierre1982}) that this Cauchy problem is well-posed and admits a unique solution. Because of the no-flux boundary conditions, $\Phi_s \mu \in \Prob(\Omega)$ for any $s \geqslant 0$. Let us summarize the properties of the heat flow that will be useful to us. 

\begin{prop}
\label{proposition_properties_heat_flow}
\
\begin{enumerate}
\item There exists $C$ (only depending on $\Omega$) such that for all $\mu \in \Prob(\Omega)$ and any $s > 0$,  
\begin{equation*} 
\| \Phi_s \mu \|_{L^\infty} \leqslant C s^{-d/2} + 1.
\end{equation*}
\item If $h : \R \to \R$ is any convex function bounded from below, $\rho \in \Prob(\Omega) \cap L^1(\Omega)$, and $s \geqslant 0$, 
\begin{equation*}
\int_\Omega h \left[ (\Phi_s \rho)(x) \right] \ddr x \leqslant \int_\Omega h[\rho(x)] \ddr x;
\end{equation*} 
If $h$ is not superlinear, the same stays true for any $\rho\in\Prob(\Omega)$ by replaing the integral $\int_\Omega  h[\rho(x)] \!\ddr x$ with the expression in \eqref{defi functional h}.
\item If $\mu$ and $\nu \in \Prob(\Omega)$, and $s \geqslant 0$, 
\begin{equation}
\label{equation_heat_flow_decreases_W}
W_2(\Phi_s \mu, \Phi_s \nu) \leqslant W_2(\mu, \nu).
\end{equation}
\item Let $\mu \in \Prob(\Omega)$ with $U_1(\mu) < + \infty$. Then the curve $s \mapsto \Phi_s \mu$ is $2$-absolutely continuous and for any $s \geqslant 0$, 
\begin{equation}
\label{equation_EI}
\int_0^s |\dot{\Phi_r \mu}|^2 \ddr r = U_1(\mu) - U_1(\Phi_s(\mu)). 
\end{equation} 
\end{enumerate}
\end{prop} 

\begin{proof}
The first point is a classic $L^\infty-L^1$ estimate for the heat equation, see for instance \cite[Section 7]{Arendt2002}. 

To prove the second point in the case of $\rho\in L^1$, let us denote by $K_t(x,y)$ the heat kernel (see  \cite[Section 7]{Arendt2002}). Using in particular Jensen's inequality and the fact that $\int_\Omega K_t(x,y) \ddr x = 1$ for any $y$ and $t$ (because $\Leb$ is invariant under the heat flow), 
\begin{align*}
\int_\Omega h \left[ (\Phi_s \rho)(x) \right] \ddr x & = \int_\Omega h \left( \int_\Omega K_s(x,y) \rho(y) \ddr y \right) \ddr x \\
& \leqslant \int_{\Omega \times \Omega}  K_s(x,y) h(\rho(y)) \ddr y \ddr x \\
& = \int_\Omega h[\rho(y)] \ddr y. 
\end{align*} 
The proof in the case where $h$ is not superlinear and $\rho$ is not absolutely continuous is obtained by writing $\rho =: \rho^{ac} \Leb + \rho^{sing} $. Observing that $h'(\infty)$ is the Lipschitz constant of $h$, we have
$$\int_\Omega h \left[ (\Phi_s \rho)(x) \right] \ddr x - \int_\Omega h \left[ (\Phi_s \rho^{ac})(x) \right] \ddr x\leq h'(\infty) \int_\Omega |\Phi_s \rho^{sing}|(x)\ddr x=h'(\infty)\rho^{sing}(\Omega).$$
The proofs of the third and last points rely on the fact that the heat flow is the gradient flow of the entropy $U_1$ in the Wasserstein space and can be found in \cite[Theorem 11.2.1]{Ambrosio2005}. 
\end{proof}  

\begin{prop}
\label{proposition_regularization_curve}
Suppose Assumption \ref{assumption_temporal_bc} holds and that $\overline{\rho_0}$ is such that $U_1(\overline{\rho_0}),F(\overline{\rho_0})<+\infty$, and let $\rho \in \Gamma$ with $\rho_0 = \overline{\rho_0}$. Then, for any $\varepsilon > 0$, there exists $\tilde{\rho} \in \Gamma$ with $\tilde{\rho}_0 = \overline{\rho_0}$ and $C < + \infty$ such that $\A(\tilde{\rho}) \leqslant \A(\rho) + \varepsilon$ and $U_1(\tilde{\rho}_t) \leqslant C$ for any $t \in \intervalleff{0}{1}$.
\end{prop}

\begin{proof}
Without loss of generality, we assume $\A(\rho) < + \infty$. The idea is to use the heat flow to regularize solutions. But we cannot apply the heat flow uniformly, as we would loose the boundary condition $\rho_0 = \overline{\rho_0}$. Consequently, for any $0 < s < T$, we define $\tilde{\rho}^s \in \Prob(\Gamma)$ by 
\begin{equation*}
\tilde{\rho}^s(t) := \begin{cases}
\Phi_t ( \rho_0) & \text{if } 0 \leqslant t \leqslant s \\
\dst{\Phi_s \left( \rho \left[ \frac{t-s}{T-s} T \right] \right)} & \text{if } s \leqslant t \leqslant T
\end{cases}.
\end{equation*}
In other words, we take the curve $\Phi_s \rho$, squeeze it into $\intervalleff{s}{T}$, and use the heat flow to join $\rho_0$ to $\rho_s$ on $\intervalleff{0}{s}$. In particular, $\tilde{\rho}^s_0 = \rho_0 = \overline{\rho_0}$ and $\tilde{\rho}^s_T = \Phi_s \rho_T$. From $U_1(\rho_0) < + \infty$ and the fact that $U_1$ is decreasing along the heat flow (see Proposition \ref{proposition_properties_heat_flow}), $U_1(\tilde{\rho}_t)$ is bounded by $U_1(\rho_0)$ if $t \in \intervalleff{0}{s}$ and by a constant $C_s$ (depending only on $s$ and $\Omega$) if $t \in \intervalleff{s}{T}$. Hence, for any $s > 0$, there exists $C < + \infty$ such that $U_1(\tilde{\rho}^s_t) \leqslant C$ for any $t \in \intervalleff{0}{1}$. 

It remains to show that $\A$ does not increase too much because of our regularization process. Using the second point of Proposition \ref{proposition_properties_heat_flow}, one can see that 
\begin{equation*}
\int_0^T  F(\tilde{\rho}^s_t)  \ddr t + G(\tilde{\rho}^s_T) \leqslant sF(\overline{\rho_0})+\frac{(T-s)}{T} \int_0^T  F(\rho_t)  \ddr t + G(\rho_T).
\end{equation*}
To handle the action of $\tilde{\rho}^s$, we remark thanks to the third point of Proposition \ref{proposition_properties_heat_flow} and the representation formula \eqref{equation_representation_A_sup} that applying uniformly the heat flow decreases the action. Hence, performing a affine change of variables on $\intervalleff{s}{T}$, 
\begin{align*}
\int_0^T | \dot{\tilde{\rho}}^s_t |^2 \ddr t & = \int_0^s | \dot{\Phi_t \rho_0} |^2 \ddr t + \frac{T}{T-s} \int_0^T | \dot{\Phi_s \rho_t} |^2 \ddr t \\
& \leqslant U_1(\rho_0) - U_1(\Phi_s \rho_0) +  \frac{T}{T-s} \int_0^T |\dot{\rho}_t|^2 \ddr t.
\end{align*}  
By lower semi-continuity of $U_1$ and as $U_1(\rho_0) = U_1(\overline{\rho_0})$ is finite, one concludes that 
\begin{equation*}
\limsup_{s \to 0} \int_0^T | \dot{\tilde{\rho}}^s_t |^2 \ddr t \leqslant  \int_0^T |\dot{\rho}_t|^2 \ddr t.  
\end{equation*}
Finally to handle the term involving the potentials, one uses, by continuity of the heat flow, that $\tilde{\rho}^s_t$ converges to $\rho_t$ for any $t \in \intervalleff{0}{1}$ as $s$ goes to $0$. As $\int_0^T |\dot{\tilde{\rho}}^s_t|^2 \ddr t$ is uniformly bounded, the family $(\tilde{\rho}^s)_{0 \leqslant s < T}$ is uniformly equicontinuous, hence $\tilde{\rho}^s$ converges uniformly to $\rho$ as $s \to 0$. This allows us to write 
\begin{equation*}
\lim_{s \to 0} \left[ \int_0^T \int_\Omega V \ddr \tilde{\rho}^s_t \ddr t + \int_\Omega W \ddr \tilde{\rho}^s_T \right] = \int_0^T \int_\Omega V \ddr \rho_t \ddr t + \int_\Omega W \ddr \rho_T.
\end{equation*}
Gluing all the inequalities that we have collected, we see that $\limsup_{s \to 0} \A(\tilde{\rho}^s) \leqslant \A(\rho)$. Hence, it is enough to take $\tilde{\rho} := \tilde{\rho}^s$ for $s$ small enough.
\end{proof}

Now, let us show how one can sample a continuous curve to get a discrete one that approximates it. 

\begin{prop}
\label{proposition_sampling}
Let $\rho \in \Gamma$ with $\rho_0 = \overline{\rho_0}$ be such that $\int_0^T U_1(\rho_t) \ddr t < + \infty$ and $\lambda > 0$ be fixed. For any $N \geqslant 1$ we can build a curve $\rho^N \in \Gamma_N$ with $\rho^N_0 = \overline{\rho_0}$ in such a way that 
\begin{equation*}
\limsup_{N \to + \infty} \A^{N, \lambda}(\rho^N) \leqslant \A(\rho) + \lambda \int_0^T U_1(\rho_t) \ddr t + \lambda_N U_1(\rho_T).
\end{equation*}
\end{prop}

\noindent We recall that $\lambda_N = 0$ by default except if Assumption \ref{assumption_temporal_bc} holds. 

\begin{proof}
We can assume $\A(\rho) < + \infty$. The idea is to sample $\rho$ on a grid translated w.r.t. $T^N$. We start with the following observation. 
\begin{align*}
\int_0^\tau \sum_{k=1}^{N-1} \left( F(\rho_{k \tau + s}) + \lambda U_1(\rho_{k \tau + s}) \right) \ddr s & = \int_0^{T-\tau} \left( F(\rho_t) + \lambda U_1(\rho_t) \right) \ddr t \\
& \leqslant \int_0^{T} \left( F(\rho_t) + \lambda U_1(\rho_t) \right) \ddr t + C \tau, 
\end{align*}
where $C$ depends only on the lower bounds of $F$ and $U_1$. Therefore, there exists $s_N \in \intervalleoo{0}{\tau}$ such that 
\begin{equation*}
\tau \sum_{k=1}^{N-1} \left( F(\rho_{k \tau + s_N}) + \lambda U_1(\rho_{k \tau + s_N}) \right) \leqslant \int_0^{T} \left( F(\rho_t) + \lambda U_1(\rho_t) \right) + C \tau.  
\end{equation*} 
Let us define $\rho^N \in \Gamma_N$ by sampling $\rho$ on the grid translated by $s_N$: for any $k \in \{ 0, 1, \ldots, N \}$,
\begin{equation*}
\rho^N := \begin{cases}
\rho_0 & \text{if } k = 0 \\
\rho_T & \text{if } k = N \\
\rho_{k \tau + s_N} & \text{if } 1 \leqslant k \leqslant N-1
\end{cases}.  
\end{equation*}
As the boundary values are left unchanged and given the choice of $s_N$, it is clear that
\begin{equation*}
\left( \A(\rho) + \lambda \int_0^T U_1(\rho_t) \ddr t + \lambda_N U_1(\rho_T) \right) - \A^{N,\lambda}(\rho^N)  \geqslant \int_0^T \frac{1}{2} |\dot{\rho}_t|^2 \ddr t - \sum_{k=1}^N \frac{W_2^2(\rho^N_{(k-1) \tau}, \rho^N_{k \tau})}{2 \tau} - C \tau.
\end{equation*} 
The r.h.s. of the above equation is delicate to evaluate because of the non uniformity of the grid near the boundaries. Recall that if $t \leqslant s$ then $W_2^2(\rho_t, \rho_s) \leqslant (s-t) \int_t^s |\dot{\rho}_r|^2 \ddr r$, hence 
\begin{align*}
\sum_{k=1}^N \frac{W_2^2(\rho^N_{(k-1) \tau}, \rho^N_{k \tau})}{2 \tau} & = \frac{W_2^2(\rho_0, \rho_{\tau + s_N})}{2 \tau} + \sum_{k=2}^{N-1} \frac{W_2^2(\rho_{(k-1) \tau + s_N}, \rho_{k \tau + s_N})}{2 \tau} + \frac{W_2^2(\rho_{(k-1) \tau + s_N}, \rho_T)}{2 \tau} \\
& \leqslant \frac{\tau + s_N}{2\tau} \int_0^{\tau + s_N} \frac{1}{2} |\dot{\rho}_t|^2 \ddr t + \sum_{k=2}^{N-1} \int_{(k-1) \tau +s_n}^{k \tau + s_N} \frac{1}{2} |\dot{\rho}_t|^2 \ddr t + \frac{\tau - s_N}{2\tau} \int_{T - \tau + s_N}^T \frac{1}{2} |\dot{\rho}_t|^2 \ddr t \\
& \leqslant \int_0^{\tau + s_N} |\dot{\rho}_t|^2 \ddr t + \int_{\tau + s_N}^{T- \tau + s_N} \frac{1}{2} |\dot{\rho}_t|^2 \ddr t + \int_{T-\tau + s_N}^{T} \frac{1}{2} |\dot{\rho}_t|^2 \ddr t \\ 
& \leqslant \int_0^T \frac{1}{2} |\dot{\rho}_t|^2 \ddr t + \int_0^{2 \tau} \frac{1}{2} |\dot{\rho}_t|^2 \ddr t.  
\end{align*}
In particular, we have used $\tau + s_N \leqslant 2 \tau$ and $\tau -s_N \leqslant \tau$. Letting $N \to + \infty$ (hence $\tau \to 0$), we end up with 
\begin{equation*}
\limsup_{N \to + \infty} \sum_{k=1}^N \frac{W_2^2(\rho^N_{(k-1) \tau}, \rho^N_{k \tau})}{2 \tau} \leqslant \int_0^T \frac{1}{2} |\dot{\rho}_t|^2 \ddr t,  
\end{equation*} 
and this is enough to conclude. 
\end{proof}

\begin{crl}
\label{corollary_uniform_bounds}
Under the assumptions of Theorems \ref{theorem_L_infty_bounds_strong}, \ref{theorem_L_infty_bounds_boundary} or \ref{theorem_L_infty_bounds_weak}, there exists $C < + \infty$ such that, uniformly in $N \geqslant 1$ and $\lambda \in \intervalleof{0}{1}$, one has 
\begin{equation*}
\A^{N, \lambda}(\bar{\rho}^{N, \lambda}) \leqslant C.
\end{equation*} 
\end{crl}

\begin{proof}
If we are under the assumptions of Theorems \ref{theorem_L_infty_bounds_strong} ot \ref{theorem_L_infty_bounds_boundary}, we take $\rho \in \Gamma$ such that $\A(\rho) < + \infty$. As $U_1 \leqslant C_f F + C$, we see that $\int_0^T U_1(\rho_t) \ddr t < + \infty$. If we are under the assumptions of \ref{theorem_L_infty_bounds_weak}, we take $\rho \in \Gamma$ such that $\A(\rho) < + \infty$ and regularize it thanks to Proposition \ref{proposition_regularization_curve}. For this regularized curve, one has $\int_0^T U_1(\rho_t) \ddr t + \lambda U_1(\rho_T) < + \infty$.  

In any of these two cases, we construct $\rho^N$ as in Proposition \ref{proposition_sampling} and define $C:=\sup_{N \geqslant 1} \A^{N,\lambda}(\rho^N)$, then we use the fact that $\A^{N, \lambda}(\bar{\rho}^{N, \lambda}) \leqslant \A^{N,\lambda}(\rho^N) \leqslant C$.  
\end{proof}

\subsection{Solution of the continuous problem as limit of discrete curves}

We will build a suitable interpolation of the discrete curves $\bar{\rho}^{N, \lambda}$ that will converge to some continuous curve $\bar{\rho}$ as $N \to + \infty$ and $\lambda \to 0$, and we will show that $\bar{\rho}$ is a solution of \eqref{equation_continuous problem}.

As the order in which the limits $N \to + \infty$ and $\lambda \to 0$ are taken does not matter, we will do them in the same time. We take two sequences $(N_n)_{n \in \N}$ and $(\lambda_n)_{n \in \N}$ that go respectively to $+ \infty$, and $0$ (the second one being strictly positive). We will not relabel the sequences when extracting subsequences. Moreover, to avoid heavy notations, we will drop the index $n$, and $\lim_{n \to + \infty}$ will be denoted by $\lim_{N \to + \infty, \lambda \to 0}$. We will need to define two kind of interpolations: one filling the gaps with constant-speed geodesics, and the other one by using piecewise constant curves.  

\begin{defi}
If $N \geqslant 1$ and $\lambda > 0$, we define $\hat{\rho}^{N, \lambda} \in \Gamma$ as the curve such that $\hat{\rho}^{N, \lambda}$ coincides with $\bar{\rho}^{N, \lambda}$ on $T^N$, and such that for any $k \in \{ 0, 1, \ldots, N-1 \}$, the restriction of $\hat{\rho}^{N, \lambda}$ to $\intervalleff{k \tau}{(k+1 \tau)}$ is the constant-speed geodesic joining $\bar{\rho}^{N, \lambda}_{k \tau}$ to $\bar{\rho}^{N, \lambda}_{(k+1) \tau}$.  
\end{defi}

\noindent As $\bar{\rho}^{N, \lambda}_{k \tau}$ is absolutely continuous w.r.t. $\Leb$ for any $k \in \{ 1,2, \ldots, N-1 \}$, the constant-speed geodesic between $\bar{\rho}^{N, \lambda}_{k \tau}$ and $\bar{\rho}^{N, \lambda}_{(k \pm 1) \tau}$ is always unique. From the characterization of constant-speed geodesics, one has, for any $k \in \{ 0, 1, \ldots, N-1 \}$, 
\begin{equation*}
\int_{k \tau}^{(k+1) \tau} \frac{1}{2} \left| \dot{\hat{\rho}}^{N, \lambda}_t \right|^2 \ddr t = \frac{W_2^2( \bar{\rho}^{N, \lambda}_{k \tau}, \bar{\rho}^{N, \lambda}_{(k+1) \tau} )}{2 \tau}.
\end{equation*}
Summing these identities over $k$, 
\begin{equation}
\label{equation_identity_action_discrete_continuous}
\int_0^T \frac{1}{2} \left| \dot{\hat{\rho}}^{N, \lambda}_t \right|^2 \ddr t = \sum_{k=1}^N \frac{W_2^2(\bar{\rho}^{N, \lambda}_{(k-1) \tau}, \bar{\rho}^{N, \lambda}_{k \tau})}{2 \tau}.
\end{equation}
In other words, the continuous action of the interpolated curve $\hat{\rho}^{N, \lambda}$ is equal to the discrete action of the discrete curve $\bar{\rho}^{N, \lambda}$. 

\begin{defi}
If $N \geqslant 1$ and $\lambda > 0$, we define $\tilde{\rho}^{N, \lambda} : \intervalleff{0}{T} \to \Prob(\Omega)$ as the function such that $\tilde{\rho}^{N, \lambda}$ coincides with $\bar{\rho}^{N, \lambda}$ on $T^N$, and such that for any $k \in \{ 0, 1, \ldots, N-1 \}$, the restriction of $\tilde{\rho}^{N, \lambda}$ to $\intervallefo{k \tau}{(k+1 \tau)}$ is equal to $\bar{\rho}^{N, \lambda}_{k \tau}$.
\end{defi}

\noindent The curve $\tilde{\rho}^{N, \lambda}$ is not continuous as it might admit discontinuities at every point in $T^N$. Let us underline that the following identity trivially holds: 
\begin{equation}
\label{equation_identity_integrals_discrete_continuous}
\sum_{k=1}^{N-1} \tau \left( F(\bar{\rho}^{N, \lambda}_{k \tau}) + \int_\Omega V \ddr \bar{\rho}^{N, \lambda}_{k \tau} \right) = \int_0^{T- \tau} \left( F(\tilde{\rho}^{N, \lambda}_t) + \int_\Omega V \ddr \tilde{\rho}^{N, \lambda}_t  \right) \ddr t
\end{equation}

\begin{prop}
Under the assumptions of Theorems \ref{theorem_L_infty_bounds_strong}, \ref{theorem_L_infty_bounds_boundary} or \ref{theorem_L_infty_bounds_weak}, there exists $\bar{\rho} \in \Gamma$ such that $\hat{\rho}^{N, \lambda}$ and $\tilde{\rho}^{N, \lambda}$ converge uniformly to $\bar{\rho}$ as $N \to + \infty$ and $\lambda \to 0$. 
\end{prop}

\begin{proof}
Let us denote by $C$ the constant given in Corollary \ref{corollary_uniform_bounds}. As all the terms in $\A^{N, \lambda}$ are bounded from below and given identity \eqref{equation_identity_action_discrete_continuous}, one can see that there exists $C_1$ such that 
\begin{equation*}
\int_0^T \frac{1}{2} \left| \dot{\hat{\rho}}^{N, \lambda}_t \right|^2 \ddr t \leqslant C_1
\end{equation*}  
uniformly in $N \geqslant 1$ and $\lambda \in \intervalleof{0}{1}$. Thus, by compactness of the sublevel sets of the action (Proposition \ref{proposition_properties_action}), one concludes of the existence of $\bar{\rho} \in \Gamma$ such that $\hat{\rho}^{N, \lambda}$ converges uniformly (up to extraction) to $\bar{\rho}$ as $N \to + \infty$ and $\lambda \to 0$. Moreover, one can see that for any $t \in \intervalleff{0}{T}$ and any $N \geqslant 1$, by setting $k$ to be the largest integer such that $k \tau \leqslant t$, one has
$$
W_2 \left( \hat{\rho}^{N, \lambda}_t, \tilde{\rho}^{N, \lambda}_t \right) = W_2 \left( \hat{\rho}^{N, \lambda}_t, \hat{\rho}^{N, \lambda}_{k \tau} \right)  \leqslant \sqrt{\tau} \sqrt{ \int_{k \tau}^t \left| \dot{\hat{\rho}}^{N, \lambda}_s \right|^2 \ddr s} \leqslant \sqrt{2 C_1 \tau}. 
$$
This allows to conclude that $\tilde{\rho}^{N, \lambda}$ also converges uniformly to $\bar{\rho}$ as $N \to + \infty$ and $\lambda \to 0$. 
\end{proof}

\begin{prop}
Under the assumptions of Theorems \ref{theorem_L_infty_bounds_strong}, \ref{theorem_L_infty_bounds_boundary} or \ref{theorem_L_infty_bounds_weak}, the curve $\bar{\rho}$ is the solution to the continuous problem \eqref{equation_continuous problem}. 
\end{prop}

\begin{proof}
Taking the limit $N \to + \infty$ and $\lambda \to 0$ in \eqref{equation_identity_action_discrete_continuous}, as the action is l.s.c., we end up with
\begin{equation*}
\int_0^T \frac{1}{2} | \dot{\bar{\rho}}_t |^2 \ddr t \leqslant \liminf_{N \to + \infty, \lambda \to 0} \sum_{k=1}^N \frac{W_2^2(\bar{\rho}^{N, \lambda}_{(k-1) \tau}, \bar{\rho}^{N, \lambda}_{k \tau})}{2 \tau}. 
\end{equation*}
Then, to handle the terms with the potential and the congestion, one can notice that for any $t \in \intervalleff{0}{T}$, by lower semi-continuity of $F$ and the convergence of $\tilde{\rho}^{N, \lambda}_t$ to $\bar{\rho}_t$, 
\begin{equation*}
F(\bar{\rho}_t) + \int_\Omega V \ddr \bar{\rho}_t \leqslant \liminf_{N \to + \infty, \lambda \to 0} F(\tilde{\rho}^{N, \lambda}_t) + \int_\Omega V \ddr \tilde{\rho}^{N, \lambda}_t.
\end{equation*}
Thus, using Fatou's lemma, as $F, V$ and $U_1$ are bounded from below, one has for any $\tau_0 > 0$, 
\begin{align*}
\int_0^{T- \tau_0} \left( F(\bar{\rho}_t) + \int_\Omega V \ddr \bar{\rho}_t  \right) & \leqslant \liminf_{N \to + \infty, \lambda \to 0} \int_0^{T- \tau} \left( F(\tilde{\rho}^{N, \lambda}_t) + \int_\Omega V \ddr \tilde{\rho}^{N, \lambda}_t  \right) \ddr t \\
& =  \liminf_{N \to + \infty, \lambda \to 0} \sum_{k=1}^{N-1} \tau \left( F(\bar{\rho}^{N, \lambda}_{k \tau}) + \int_\Omega V \ddr \bar{\rho}^{N, \lambda}_{k \tau}  \right) \ddr t \\
& \leqslant \liminf_{N \to + \infty, \lambda \to 0} \sum_{k=1}^{N-1} \tau \left( F(\bar{\rho}^{N, \lambda}_{k \tau}) + \int_\Omega V \ddr \bar{\rho}^{N, \lambda}_{k \tau} + \lambda U_1(\bar{\rho}^{N, \lambda}_{k \tau})  \right).
\end{align*}
In the equation above, $\tau_0$ is arbitrary thus it is still valid for $\tau_0 = 0$. As moreover the boundary penalization $\Psi$ is l.s.c. and the entropic penalization $\lambda_N U_1(\rho_T)$ is positive, one is allowed to write that 
\begin{equation*}
\A(\bar{\rho}) \leqslant \liminf_{N \to + \infty, \lambda \to 0} \A^{N, \lambda}( \bar{\rho}^{N, \lambda} ). 
\end{equation*} 

Let us assume by contradiction that there exists $\rho \in \Gamma$ such that $\A(\rho ) < \A(\bar{\rho})$. Using, if needed, Proposition \ref{proposition_regularization_curve}, we can assume without loss of generality that $\A(\rho) < \A(\bar{\rho})$ and $\int_0^T U_1(\rho_t) \ddr t + \lambda_N U_1(\rho_T) < + \infty$. Using Proposition \ref{proposition_sampling}, for any $N \geqslant 1$, we can build $\rho^N \in \Gamma_N$ in such a way that 
\begin{equation*}
\limsup_{N \to + \infty} \A^{N, \lambda}(\rho^N) \leqslant \A(\rho) + \lambda \int_0^T U_1(\rho_t) \ddr t + \lambda_N U_1(\rho_T).
\end{equation*}
Taking the limit $\lambda \to 0$, one can see that 
$$
\limsup_{N \to + \infty, \lambda \to 0} \A^{N, \lambda}(\rho^N)  \leqslant \A(\rho)< \A(\bar{\rho}) \leqslant \liminf_{N \to + \infty, \lambda \to 0} \A^{N, \lambda}( \bar{\rho}^{N, \lambda} ). 
$$
Taking $N$ large enough and $\lambda$ small enough, we conclude that $\A^{N, \lambda}(\rho^N) < \A^{N, \lambda}(\bar{\rho}^{N, \lambda})$, which is a contradiction with the optimality of $\bar{\rho}^{N, \lambda}$.
\end{proof}

\subsection{Uniform bounds on \texorpdfstring{$\bar{\rho}$}{rho}}

To conclude and prove the Theorems \ref{theorem_L_infty_bounds_strong}, \ref{theorem_L_infty_bounds_boundary} and \ref{theorem_L_infty_bounds_weak}, it is enough to show the $L^\infty$ bounds on $\bar{\rho}$, which of course we will prove using the discrete solutions $\bar{\rho}^{N, \lambda}$. The key is the following proposition. 

\begin{prop}
\label{proposition_norm_discrete_to_continuous}
Let $0 < T_1 < T_2 \leqslant T$. Then for any $0 <T'_1<T_1 $ and any $T_2 < T'_2 < T$ (or $T'_2 = T_2 = T$ in the case $T_2 = T$),
\begin{equation*}
\esssup_{T_1 \leqslant t \leqslant T_2,  x \in \Omega} |\bar{\rho}_t(x)| \leqslant \limsup_{m \to + \infty} L^m_{T'_1, T'_2 }.
\end{equation*}
\end{prop}

\begin{proof}
We rely on the well-known identity 
$$
\esssup_{T_1 \leqslant t \leqslant T_2, \ x \in \Omega} |\bar{\rho}_t(x)|  = \limsup_{m \to + \infty} \left( \int_{T_1}^{T_2} \int_\Omega \bar{\rho}_t^m \ddr t \right)^{1/m}  = \limsup_{m \to + \infty} \left( \int_{T_1}^{T_2} U_m(\bar{\rho}_t) \ddr t \right)^{1/m}.
$$
For a fixed $m>1$ and for $\tau > 0$ small enough, one has 
\begin{equation*}
\int_{T_1}^{T_2} U_m(\tilde{\rho}^{N, \lambda}_t) \ddr t  \leqslant \sum_{T'_1 \leqslant k \tau \leqslant T'_2} \tau U_m(\bar{\rho}^{N, \lambda}_{k \tau}).  
\end{equation*} 
When sending $N \to \infty$ and $\lambda \to 0$, by lower semi-continuity of $U_m$ and by convergence of $\tilde{\rho}^{N, \lambda}$ to $\bar{\rho}$, we know that 
\begin{align*}
\int_{T_1}^{T_2} U_m(\bar{\rho}_t) \ddr t & \leqslant \liminf_{N \to + \infty, \lambda \to 0} \int_{T_1}^{T_2} U_m(\tilde{\rho}^{N, \lambda}_t) \ddr t \\
& \leqslant \liminf_{N \to + \infty, \lambda \to 0} \sum_{T'_1 \leqslant k \tau \leqslant T'_2} \tau U_m(\bar{\rho}^{N, \lambda}_{k \tau}).
\end{align*} 
Taking the power $1/m$ on each side and by definition of $L^m_{T'_1, T'_2 }$, one gets 
\begin{equation*}
\left( \int_{T_1}^{T_2} U_m(\bar{\rho}_t) \ddr t \right)^{1/m} \leqslant L^m_{T'_1 , T'_2 }. 
\end{equation*}
It is enough to take the limit $m \to + \infty$ to get the announced inequality.
\end{proof}

We can now conclude the desired bounds:

\begin{proof}[Proof of Theorem \ref{theorem_L_infty_bounds_strong}]
Combining Proposition \ref{proposition_norm_discrete_to_continuous} and Proposition \ref{proposition_conclusion_iteration_flow_interchange}, it is enough to show that $L^{\alpha + 2}_{0,T} < + \infty$. Because of Assumption \ref{assumption_strong_congestion} or \ref{assumption_strong_congestion-variant}, we know that $U_{\alpha + 2} \leqslant C_1 F + C_2$ with $C_1 > 0$. Hence, in order to conclude that $L^{\alpha + 2}_{0,T} < + \infty$, it is enough to use Corollary \ref{corollary_uniform_bounds}, which provides a constant $C < + \infty$ such that for any $N \geqslant 1$ and any $\lambda \in \intervalleof{0}{1}$ we have
\begin{equation*}
\sum_{k=1}^{N-1} \tau F(\bar{\rho}^{N, \lambda}_{k \tau}) \leqslant C.\qedhere
\end{equation*} 
\end{proof}

\begin{proof}[Proof of Theorem \ref{theorem_L_infty_bounds_boundary}]
We combine Proposition \ref{proposition_norm_discrete_to_continuous} and Proposition \ref{proposition_conclusion_iteration_flow_interchange_boundary}, as we saw that $L^{\alpha + 2}_{0,T} < + \infty$ (in the proof of Theorem \ref{theorem_L_infty_bounds_strong}).  
\end{proof}


\begin{proof}[Proof of Theorem \ref{theorem_L_infty_bounds_weak}]
It is enough to combine Proposition \ref{proposition_norm_discrete_to_continuous} with Propositions \ref{proposition_conclusion_iteration_flow_interchange_weak} and \ref{proposition_initialisation_weak_congestion}.
\end{proof}

\appendix

\section{Reverse Jensen inequality}
\label{appendix}

In this section, we prove Lemma \ref{lemma_reverse_Jensen} (the "reverse Jensen inequality") as well as Lemmas \ref{lemma_reverse_Jensen_Neumann} and \ref{lemma_reverse_Jensen_boundary}, whose proofs were postponed in order not to overload the key arguments of the paper. In all the sequel, we consider a family of sequences $(u_k^\tau)_{k \in \Z}$ indexed by a parameter $\tau > 0$. We assume that there exists $\omega \geqslant 0$ such that for any $k \in \Z$, one has $u_k^\tau > 0$ and
\begin{equation}
\label{equation_appendix_almost_convexity}
\frac{u^\tau_{k+1} + u^\tau_{k-1} - 2 u^\tau_{k}}{\tau^2} + \omega^2 u^\tau_k \geqslant 0.
\end{equation} 
This inequation is a discrete counterpart of the differential inequality $u'' + \omega^2 u \geqslant 0$. Let us remark, by the positivity of $u_k^\tau$, that we can assume without loss of generality that $\omega > 0$, even though the proofs are considerably simpler if $\omega = 0$: the constants would be better, and the strategy of the proof would be slightly different. The key point to handle $u_k^\tau$ is to compare it with explicit sequences realizing the opposite inequality in \eqref{equation_appendix_almost_convexity}. 

\begin{defi}
For any $\tau > 0$, let $\T^\tau$ be the set of sequences $(v_k)_{k \in \Z}$ of the form $v_k = A \cos( 2 \omega k \tau + \delta )$.
\end{defi}  

\begin{lm}
\label{lemma_v_subsolution}
There exists $\tau_0 > 0$ such that for any $\tau \leqslant \tau_0$, if $(v_k)_{k \in \Z} \in \T^\tau$ and $k$ is such that $v_k > 0$ then 
\begin{equation*}
\frac{v_{k+1} + v_{k-1} - 2 v_{k}}{\tau^2} + \omega^2 v_k < 0
\end{equation*} 
\end{lm}

\begin{proof}
This is a consequence of the trigonometric identity 
\begin{equation*}
\frac{v_{k+1} + v_{k-1} - 2 v_{k}}{\tau^2} + \omega^2 v_k = \left( 2 \frac{\cos(2 \omega \tau) - 1}{\tau^2} + \omega^2 \right) v_k
\end{equation*} 
and the fact that $\dst{2 \frac{\cos(2 \omega \tau) - 1}{\tau^2} + \omega^2} \sim - 3 \omega^2 $ as $\tau$ goes to $0$.
\end{proof}

We also note the following properties on the sequences in $\T^\tau$, that we do not prove and leave to the reader as an exercise.
\begin{itemize}
\item if $k_1 < k_2$ are fixed with $|k_2 -k_1| \tau \omega< \pi / 8$ and $\tau$ is small enough, then for every fixed positive values $a_1,a_2 > 0$ there exists a unique sequence in $\T^\tau$ with $v_{k_1}=a_1$ and $v_{k_2}=a_2$. Moreover, such a sequence $(v_k)_{k \in \Z}$ is such that there exists an open interval $I$ of the form either $(k_0\tau,k_1\tau)$ or $(k_2\tau,k_3\tau)$, with length at least $\pi/(8\omega)$, with $v_k>0$ for all the indices $k$ such that $k\tau\in I$.
\item if $k_1<N$ and $b \geqslant 0$ are fixed and $|N-k_1|\tau< \min \{ \pi / (8 \omega), \pi / (8 b) \}$ and $\tau$ is small enough, then for every $a>0$ there exists a unique sequence in $\T^\tau$ with $v_{k_1}=a$ and $(v_N-v_{N-1})/\tau=b v_N$. Moreover, such a sequence $(v_k)_{k \in \Z}$ is such that there exists an open interval $I$ of the form $(k_0\tau,k_1\tau)$ with length at least $\min\{\pi/(32\omega),\pi/(32 b)\}$, with $v_k>0$ for all the indices $k$ such that $k\tau\in I$.
\end{itemize}

Note that, for the purpose of Lemma \ref{lemma_v_subsolution} and of the subsequent observations other choices of $v_k$ were possible, such as $v_k = A \cos( (1+\varepsilon) \omega k \tau + \delta )$ for some $\varepsilon>0$, but we chose $\varepsilon=1$ for simplicity in the next computations (more generally in this appendix we have not been looking for the sharpest constants). Indeed, all these results are not surprising: at the continuous level $v$ solves $v'' + 4 \omega^2 v = 0$ and most of the discrete results are just an adaptation of this property. The important point is the following comparision principle between $(u^\tau_k)_{k \in \Z}$ and $(v_k)_{k \in \Z}$.
\begin{lm}
\label{lemma_comparision_u_v_Dirichlet}
Let $k_1 < k_2$ such that $|k_2 -k_1| \tau \omega< \pi / 8$ and assume $\tau \leqslant \tau_0$. Let $v \in \T^\tau$ the unique element of $\T^\tau$ such that $v_{k_1} = u_{k_1}^\tau$ and $v_{k_2} = u_{k_2}^\tau$. Let $k_0$ (resp. $k_3$) be the largest (resp. the smallest) index smaller that $k_1$ (resp. larger than $k_2$) such that $v_{k_0 - 1} < 0$ (resp. $v_{k_3 + 1} < 0$). 

Then $u_k^\tau \leqslant v_k$ for any $k_1 \leqslant k \leqslant k_2$ and $u_k^\tau \geqslant v_k$ for any $k_0 \leqslant k \leqslant k_1$ and any $k_2 \leqslant k \leqslant k_3$. 
\end{lm} 

\noindent In other words, $u^\tau$ is below $v$ between $k_1$ and $k_2$ and above outside $k_1$ and $k_2$ (as long as $v \geqslant 0$).

\begin{proof}
The fact that there exists only one $v \in \T^\tau$ such that $v_{k_1} = u_{k_1}^\tau$ and $v_{k_2} = u_{k_2}^\tau$ has been already observed above. Let us define $w_k = u^\tau_k-v_k$. By \eqref{equation_appendix_almost_convexity} and Lemma \ref{lemma_v_subsolution}, 
\begin{equation}
\label{equation_w_sursolution}
\frac{w_{k+1} + w_{k-1} - 2 w_k}{\tau^2} + \omega^2 w_k > 0
\end{equation}   
for any $k_0 \leqslant k \leqslant k_3$ and $w_{k_1} = w_{k_2} = 0$.  We want to prove $w_k\leq 0$ for every $k_1\leq k\leqslant k_2$. We consider the piecewise affine interpolation $\bar w$ of the values $w_k$: a function which is affine on each interval $[k\tau,(k+1)\tau]$ and is equal to $w_k$ at the point $k\tau$. The condition \eqref{equation_w_sursolution} translates on $\bar w$ as differential inequality in the sense of distributions:
\begin{equation}\label{distrib bar w}
\bar w''+\omega^2\sum_k \tau w_k\delta_{k\tau}\geq 0.
\end{equation}

Let us assume by contradiction that there is an open interval $I\subset (k_1\tau,k_2\tau)$ on which $\bar w>0$, with $\bar w=0$ on $\partial I$. We denote by $|I|$ the length of such an interval, and we have $|I|\leq |k_2-k_1|\tau$. By multiplying the above inequality by $\bar w$ and integrating by parts we get
$$\int_I |\bar w'|^2=-\int_I \bar w'' \bar w \leq \omega^2\sum_{k\ : \ k\tau\in I} \tau |w_k|^2.$$
Then, we observe that we have, for each $k$ s.t. $k\tau\in I$, 
\begin{equation*}
|w_k|\leq \frac 12 \int_I |\bar w'|\leq \frac12 \sqrt{|I|\int_I  |\bar w'|^2}.
\end{equation*}
The reason for the factor $1/2$ in the above inequality is the possibility to choose to integrate $\bar w'$ on an interval at the right or at the left of $k\tau$, and to choose the one where the integral of $|\bar w'|$ is smaller.
This implies
$$\int_I |\bar w'|^2\leq \omega^2\tau \;\#\{k:k\tau\in I\}\frac 14 |k_2-k_1|\tau\int_I  |\bar w'|^2.$$
Since $\{k:k\tau\in I\}\subset\{k:k_1<k<k_2\}$, we have $ \#\{k:k\tau\in I\}<|k_2-k_1|$ and the contradiction comes from the assumption $\omega\tau|k_2-k_1|<\pi/8 < 2$.

In order to prove $w_k\geq 0$ for $k_0\leq k\leq k_1$, we first observe that \eqref{equation_w_sursolution} for $k=k_1$, now that we know $w_{k_1+1}\leq 0$, implies $w_{k_1-1}> 0$. If for some $k$ with $k_0\leq k\leq k_1$ we had $w_k<0$, then we could find an open interval $J\subset (k_0\tau,k_1\tau)$ where $\bar w>0 $ with $\bar w=0$ on $\partial J$. We then apply the same approach as above, thus obtaining
$$\int_J |\bar w'|^2\leq \omega^2\tau \; \#\{k:k\tau\in J\}\frac 14|J|\int_J  |\bar w'|^2.$$
It is important to not that $J$ is contained in an interval of positivity of a function of the form $A\cos(2\omega t+\delta)$, whose length is $\pi/(2\omega)$; the number of points of the form $k\tau$ contained in an interval of such a length is at most $\pi/(2\omega\tau)+1$ but for $k=k_1, k_2$ the point $k\tau$ does not belong to the open interval $J$. Hence $\#\{k:k\tau\in J\}\leq \pi/(2\omega\tau)$, and we have a contradiction since $\pi^2<16$.
\end{proof}

We provide now a variant in the case where on the interval $(k_1\tau,k_2\tau)$ we impose a different boundary condition on the right end side.

\begin{lm}
\label{lemma_comparision_u_v_Neumann}
Let $k_1 < N$ and $b \geqslant 0$ such that $|N-k_1| \tau < \min \{ \pi / (8 \omega), \pi \ (8 b) \}$ and assume $\tau \leqslant \tau_0$. Suppose $(u_{N}-u_{N-1})/\tau\leq b u_N$. Let $v \in \T^\tau$ the unique element of $\T^\tau$ such that $v_{k_1} = u_{k_1}^\tau$ and $(v_{N}-v_{N-1})/\tau= bv_N$.  Let $k_0$ be the largest (resp. the smallest) index smaller that $k_1$such that $v_{k_0 - 1} < 0$.

Then $u_k^\tau \leqslant v_k^\tau$ for any $k_1 \leqslant k \leqslant N$ and $u_k^\tau \geqslant v_k^\tau$ for any $k_0 \leqslant k \leqslant k_1$. 
\end{lm} 
\begin{proof}
The argument is very similar to the one in Lemma \ref{lemma_comparision_u_v_Dirichlet}. We first define $w_k=u_k-v_k$, as well as the piecewise affine interpolation $\bar w$ of the values $w_k$, which satisfies again \eqref{distrib bar w}, but also $w'(T)\leq bw(T)$, where $T=N\tau$. 

Then, we assume by contradiction that there is an open interval $I\subset (k_1\tau,N\tau)$ on which $\bar w>0$. If $\bar w=0$ on $\partial I$ (i.e., on both points on the boundary), the argument is really the same. Otherwise, we can assume $I=(t,T)$, with $\bar w(t)=0$. By multiplying by $\bar w$ and integrating by parts we get
$$\int_I |\bar w'|^2=\bar w(T)\bar w'(T)-\int \bar w'' \bar w \leq b|\bar w(T)|^2+ \omega^2\sum_{k\ : \ k\tau\in I} \tau |w_k|^2.$$
Then, we use that on $I$ we have
\begin{equation*}
|\bar w|\leq \int_I |\bar w'|\leq \sqrt{|I|\int_I  |\bar w'|^2}.
\end{equation*}
We do not have anymore the factor $1/2$ because $\bar w$ only vanishes at one end, now.  This implies
$$\int_I |\bar w'|^2\leq|I|\left( \omega^2\tau \;\#\{k:k\tau\in I\}+b\right)\int_I  |\bar w'|^2.$$
Since  $ \#\{k:k\tau\in I\}<|N-k_1|$ and $|I|\leq |N-k_1|\tau$, using the assumptions on $|N-k_1|$ we have 
$$\int_I |\bar w'|^2\leq\left( \frac{\pi}{8}+\left( \frac{\pi}{8}\right)^2\right)\int_I  |\bar w'|^2.$$
This is a contradiction, since
\begin{equation*}
\frac{\pi}{8}+\left( \frac{\pi}{8}\right)^2<\frac 12+\frac 14<1.\qedhere
\end{equation*}
\end{proof}

With the two lemma above, we are able to deduce some Harnack-type inequality, which means that we can control the values of a $u$ satisfying \eqref{equation_appendix_almost_convexity} in the interior of an interval with the values of $u$ outside the interval. 

\begin{lm}
\label{lemma_Harnack}
Let $k_1 < k_2$ such that $|k_2 -k_1| \tau \omega< \pi / 8$ and assume $\tau \leqslant \tau_0$. Let $k_0$ (resp. $k_3$) be the smallest (resp. largest) integer smaller than $k_1$ (resp. larger than $k_2$) such that $(k_1 -k_0) \tau \omega < \pi /8$ (resp. $(k_3 -k_2) \tau \omega < \pi /8$). Then one has
\begin{equation*}
\sup_{k_1 \leqslant k \leqslant k_2} u_k^\tau \leqslant C \max \left( \inf_{k_0 \leqslant k \leqslant k_1} u_k^\tau \ , \ \inf_{k_2 \leqslant k \leqslant k_3} u_k^\tau  \right), 
\end{equation*}
where the constant $C$ is universal. 
\end{lm}

\begin{proof}
Given the symmetry of the property we want to prove w.r.t. to time reversal, we can assume that $u^\tau_{k_1} \leqslant u^\tau_{k_2}$. Let $v \in \T^\tau$ be the unique element of $\T^\tau$ such that $v_{k_1} = u_{k_1}^\tau$ and $v_{k_2} = u_{k_2}^\tau$. We know that it can be written in the form $v_k = A \cos(k \tau \omega + \delta)$ with $A \geqslant 0$. In particular, $A \geqslant |v_k|$ for any $k \in \Z$. Up to a time translation, we can assume that $\delta = 0$ and $k_1 \leqslant 0 \leqslant k_2$. By the hypothesis $u_{k_1}^\tau \leqslant u_{k_2}^\tau$, and $|k_2 -k_1| \tau \omega< \pi / 8$, we can even say that $|k_2| \leqslant |k_1|$; thus, one has $k_2 \tau \leqslant \pi /(16 \omega)$. In particular, for any $k_2 \leqslant k \leqslant k_3$, we can say more than $v_k > 0$:
\begin{align*}
v_k & \geqslant A \cos \left( 2 \omega k_3 \tau \right) \\
& \geqslant A \cos \left( 2 k_2 \omega \tau + 2(k_3 - k_2) \omega \tau \right) \\
& \geqslant \cos \left( \frac{\pi}{8} + \frac{\pi}{4} \right) \sup_{k' \in \Z} |v_{k'}| \\
& \geqslant \frac{1}{C} \sup_{k' \in \Z} |v_{k'}|,  
\end{align*}   
with $C = \cos(3 \pi / 8)^{-1} < + \infty$. Thus, by using the comparison between $u^\tau$ and $v$ (Lemma \ref{lemma_comparision_u_v_Dirichlet}), one can say that, for any $k_2 \leqslant k  \leqslant k_3$, 
\begin{equation*}
u_k^\tau \geqslant \frac{1}{C} \sup_{k_1 \leqslant k' \tau \leqslant k_2} u_{k'}^\tau,
\end{equation*} 
which easily implies the claim. 
\end{proof}

We also provide the same type of lemma but where a different condition is imposed on the right end side, namely a Neumann-type boundary condition.

\begin{lm}
\label{lemma_Harnack_Neumann}
Let $k_1 < N$ and $b \geqslant 0$ such that $|N-k_1| \tau \leqslant \min\{ \pi / (32 \omega), \pi / (32 b) \}$ and assume $\tau \leqslant \tau_0$. Suppose $(u_{N}-u_{N-1})/\tau\leq b u_N$. Let $k_0$ be the smallest integer smaller than $k_1$ such that $(k_1 - k_0) \tau \leqslant \min \{ \pi /(32 \omega), \pi /(32b) \}$. Then one has 
\begin{equation*}
\sup_{k_1 \leqslant k \leqslant N} u^\tau_k \leqslant C \inf_{k_0 \leqslant k \leqslant k_1} u^\tau_k,
\end{equation*}
where the constant $C$ is universal.
\end{lm}

\begin{proof}
The strategy of the proof is the same than for Lemma \ref{lemma_Harnack}. We take $v$ to be the unique element of $\T^\tau$ such that $v_{k_1} = u_{k_1}^\tau$ and $(v_{N}-v_{N-1})/\tau= bv_N$. We know that $v$ is of the form $v_k = A \cos(2 k \tau \omega + \delta )$. Up to a time translation, we can assume that $N \tau = 0$ and take $\delta \in \intervalleoo{- \pi /2}{ \pi / 2}$. Starting from $(v_{N}-v_{N-1})/\tau= bv_N$ and using well known factorization formulas, one ends up with 
\begin{equation*}
b = - 2\omega \tan(\delta) + O(\omega \tau).
\end{equation*} 
Thus, if $\tau \leqslant \tau_0$, one can say that $\arctan(- b / \omega) \leqslant \delta \leqslant \arctan(-b/(4 \omega))$. Hence, using the fact that $\arctan(t) + \arctan(1/t) = - \pi / 2$ (if $t < 0$) and that $\min \{ \pi t / 4, \pi /4 \} \leqslant \arctan(t) \leqslant t$ (if $t \geqslant 0$), one concludes that 
\begin{equation*}
\min \left\{ - \frac{\pi}{2} + \frac{\pi \omega}{4 b}, - \frac{\pi}{4} \right\} \leqslant \delta \leqslant \min \left\{ - \frac{\pi}{2} + \frac{4\omega}{b}, 0 \right\}.  
\end{equation*}
In other words, $\delta$ cannot be too close to $- \pi /2$, the point where the cosine vanishes. Given the information that we have on $k_1$ and $k_0$, one can check that 
\begin{align*}
\delta - 2 \omega \tau k_0 & = \delta - 2 \omega \tau k_1 - 2 \omega \tau (k_0 - k_1) \\
& \geqslant \min \left\{ - \frac{\pi}{2} + \frac{\pi \omega}{4 b}, - \frac{\pi}{4} \right\} - 2 \min \left\{ \frac{\pi}{16}, \frac{\pi \omega}{16 b} \right\} \\
& \geqslant \min \left\{ - \frac{\pi}{2} + \frac{\pi \omega}{8 b}, - \frac{3 \pi}{8} \right\}.
\end{align*} 
As, for every $k_0 \leqslant k \leqslant N$, one has 
\begin{equation*}
A \cos(\delta - 2 \omega \tau k_0) \leqslant v_k \leqslant A \cos(\delta),
\end{equation*}
it is easy to conclude that
\begin{equation*}
\frac{\sup_{k_0 \leqslant k \leqslant N} v_k}{\inf_{k_0 \leqslant k \leqslant N} v_k}  \leqslant \frac{\cos( \min \{ - \frac{\pi}{2} + \frac{4\omega}{b}, 0\} )}{  \cos( \min \{- \frac{\pi}{2} + \frac{\pi \omega}{8b}, \frac{3 \pi}{8} \} )  } \leqslant C,
\end{equation*} 
where the value of $C$ can be estimated by noting that if $\omega/b\ll 1$ both the numerator and the denominator are of the order of $\omega/b$ and if $\omega/b$ is not small the denominator is far from $0$ and the numerator is bounded by $1$. This proves that $C$ is a universal constant. It remains to use Lemma \ref{lemma_comparision_u_v_Neumann} to transfer the above inequality into an information on $u^\tau$.  
\end{proof}

To conclude, we can prove the Lemmas \ref{lemma_reverse_Jensen}, \ref{lemma_reverse_Jensen_Neumann} and \ref{lemma_reverse_Jensen_boundary} that we used throughout the paper, by using the above results. To prove Lemma \ref{lemma_reverse_Jensen}, we cut the interval $\intervalleff{T_1}{T_2}$ into several pieces of length of order $1 / \omega$, on each piece we use the Harnack inequality to exchange the sum and the power $1 / \beta$, and we use rough comparisons to put the pieces together. 

\begin{proof}[Proof of Lemma \ref{lemma_reverse_Jensen}]
Let $M$ be the smallest integer larger than $8 \omega (T_2 - T_1)/ \pi +1$. We cut the interval $\intervalleff{T_1}{T_2}$ into $M$ closed intervals $I_1, I_2, \ldots, I_M$ of equal length (all equal to $(T_2 - T_1)/M < \pi / (8 \omega)$). Let us choose an interval $I_i$, we can use Lemma \ref{lemma_Harnack} to write
\begin{align*}
\left( \sum_{k \ : \ k \tau \in I_i} \tau u_k^\tau \right)^{1/\beta} & \leqslant  (|I_i| + \tau)^{1/\beta} \sup_{k \tau \in I_i} (u^k_\tau)^{1 / \beta} \\
& \leqslant C (|I_i| + \tau)^{1/\beta} \left( \inf_{T_1^i - \eta \leqslant k \tau \leqslant T_1^i} (u_k^\tau)^{1/\beta} + \inf_{T_2^i \leqslant k \tau \leqslant T_2^i + \eta} (u_k^\tau)^{1/\beta}  \right) \\
& \leqslant C \frac{(|I_i| + \tau)^{1/\beta}}{\eta} \sum_{k \ : \ k \tau \in I_i \pm \eta} (u_k^\tau)^{1 / \beta}, 
\end{align*} 
where $I_i \pm \eta$ denotes the set of real numbers which are at a distance at most $\eta$ of $I_i$. Then we put together the estimate for each $I_i$:
\begin{align*}
\left( \sum_{T_1 \leqslant k \tau \leqslant T_2} \tau u_k^\tau \right)^{1/ \beta} & \leqslant \left( \sum_{i=1}^M \sum_{k \ : \ k \tau \in I_i} \tau u_k^\tau \right)^{1/ \beta} \\
& \leqslant M^{1/ \beta} \sum_{i=1}^M \left( \sum_{k \ : \ k \tau \in I_i} \tau u_k^\tau \right)^{1/ \beta} \\
& \leqslant C M^{1/ \beta} \left( \frac{T_2 - T_1}{M} + \tau \right)^{1/ \beta} \frac{1}{\eta} \sum_{i=1}^M \sum_{k \ : \ k \tau \in I_i \pm \eta } \tau \left( u_k^\tau \right)^{1/ \beta} \\
& \leqslant C \frac{M(T_2 - T_1 + M\tau)^{1/ \beta}}{\eta} \sum_{T_1 - \eta \leqslant k \tau \leqslant T_2 + \eta} \tau \left( u_k^\tau \right)^{1/ \beta} \\
& \leqslant C \frac{(\omega + 1) (T_2 - T_1 + 1)^{1+1/ \beta}}{\eta}  \sum_{T_1 - \eta \leqslant k \tau \leqslant T_2 + \eta} \tau \left( u_k^\tau \right)^{1/ \beta},   
\end{align*}
where we have used the fact that $M \tau \leqslant 1$ if $\tau \leqslant \tau_0$ (where $\tau_0$ depends on $\omega$) and also that $M$ can be estimated by a constant times $\omega+1$.
\end{proof} 

\begin{proof}[Proof of Lemma \ref{lemma_reverse_Jensen_Neumann}]
For the first part, we apply Lemma \ref{lemma_Harnack_Neumann} with $k_1 = N$. With the choice of $\eta$, one has $(k_1-k_0) \tau \leqslant \min \{ \pi /(32 \omega), \pi /(32b) \}$. Thus, one can write that 
\begin{equation*}
u_N^\tau \leqslant C \inf_{T - \eta \leqslant k \tau \leqslant T} u^\tau_k,
\end{equation*}
which is enough to to conclude as the r.h.s. is bounded by the mean of $u^\tau_k$, for $T - \eta \leqslant k \tau \leqslant T$. 

For the second part (which is a variant of Lemma \ref{lemma_reverse_Jensen}, but with Neumann boundary conditions on one side), we can say with the help of Lemma \ref{lemma_Harnack_Neumann} that with $k_1$ the smallest integer smaller than $N$ such that $|N-k_1| \tau \max\{\omega,b \} < \pi / 32$, 
\begin{align*}
\left( \sum_{k_1 \tau \leqslant k \tau \leqslant T} \tau u_k^\tau \right)^{1/ \beta} & \leqslant |T - k_1 \tau + \tau|^{1/\beta} \sup_{k_1 \tau \leqslant k \tau \leqslant T} (u_k^\tau)^{1/ \beta} \\
& \leqslant C |T - k_1 \tau + \tau|^{1/\beta} \inf_{k_1 \tau - \eta \leqslant k \tau \leqslant k_1 \tau} (u_k^\tau)^{1 / \beta} \\
& \leqslant \frac{C |T - k_1 \tau + \tau|^{1/\beta}}{\eta} \sum_{k_1 \tau - \eta \leqslant k \tau \leqslant k_1 \tau} \tau (u_k^\tau)^{1 / \beta}.   
\end{align*}
Then, we combine this estimate with the interior estimate Lemma \ref{lemma_reverse_Jensen} (with $T_2 = T - k_1 \tau$) to end up with the announced result. 
\end{proof}

\begin{proof}[Proof of Lemma \ref{lemma_reverse_Jensen_boundary}]
We apply Lemma \ref{lemma_Harnack_Neumann} with $k_1 = 0$. Thus if $T = kN \leqslant \min \{ \pi / (32 \omega), \pi / (32b) \}$, one has 
\begin{equation*}
\sup_{0 \leqslant k \tau \leqslant T} u^\tau_k \leqslant C  u^\tau_{0} = Ca. 
\end{equation*}
Thus, the l.h.s. is bounded by a constant which does not depend on $N$.  
\end{proof}

\section*{Acknowledgments}

Both authors acknowledge the support of the French ANR via the contracts ISOTACE (ANR-12-MONU-0013) and MFG (ANR-16-CE40-0015-01) and benefited from the support of the FMJH ``Program Gaspard Monge for optimization and operations research and their interactions with data science'' and EDF via the PGMO project  VarPDEMFG.


\addcontentsline{toc}{section}{References}
\bibliographystyle{plain}
\bibliography{bibliography}

\end{document}